\documentclass{ejm}

\usepackage{graphicx}
\usepackage{amstext}
\usepackage{bm}
\usepackage{mathrsfs}
\usepackage{amsmath}
\usepackage{color}

\let\realverbatim=\verbatim
\let\realendverbatim=\endverbatim
\renewcommand\verbatim{\par\addvspace{6pt plus 2pt minus 1pt}\realverbatim}
\renewcommand\endverbatim{\realendverbatim\addvspace{6pt plus 2pt minus 1pt}}

\ifprodtf \else
  \checkfont{eurm10}
  \iffontfound
    \IfFileExists{upmath.sty}
      {\typeout{^^JFound AMS Euler Roman fonts on the system,
                   using the 'upmath' package.^^J}%
       \usepackage{upmath}}
      {\typeout{^^JFound AMS Euler Roman fonts on the system, but you
                   don't seem to have the}%
       \typeout{'upmath' package installed. EJM.cls can take advantage
                 of these fonts,^^Jif you use 'upmath' package.^^J}%
       \providecommand\mathup[1]{##1}%
       \providecommand\mathbup[1]{##1}%
      }
  \else
    \providecommand\mathup[1]{#1}%
    \providecommand\mathbup[1]{#1}%
  \fi
\fi

\ifprodtf \else
  \checkfont{msam10}
  \iffontfound
    \IfFileExists{amssymb.sty}
      {\typeout{^^JFound AMS Symbol fonts on the system, using the
                'amssymb' package.^^J}%
       \usepackage{amssymb}%
       \let\le=\leqslant  \let\leq=\leqslant
       \let\ge=\geqslant  \let\geq=\geqslant
      }{}
  \fi
\fi

\ifprodtf \else
  \IfFileExists{amsbsy.sty}
    {\typeout{^^JFound the 'amsbsy' package on the system, using it.^^J}%
     \usepackage{amsbsy}}
    {\providecommand\boldsymbol[1]{\mbox{\boldmath $##1$}}}
\fi

\newcommand{\f}{\frac}

\newcommand{\ds}{\displaystyle}

\newcommand{\va}{\varsigma}
\newsavebox{\astrutbox}
\sbox{\astrutbox}{\rule[-5pt]{0pt}{20pt}}

\newtheorem{theorem}{Theorem}[section]
\newdefinition{definition}[theorem]{Definition}
\newtheorem{lemma}[theorem]{Lemma}

\newtheorem{proposition}[theorem]{Proposition}
\newtheorem{remark}[theorem]{Remark}

\title[European Journal of Applied Mathematics]{Dynamics of a delayed population patch model with the dispersion matrix incorporating population loss}

\author[D. Huang and S. Chen]{%
  D\ls A\ls N\ns H\ls U\ls A\ls N\ls G$\,^1$\ns
\and
  S\ls H\ls A\ls N\ls S\ls H\ls A\ls N\ns C\ls H\ls E\ls N$\,^2$
}

\affiliation{%
  $^1\,$School of Mathematics, Harbin Institute of Technology, Harbin, Heilongjiang, 150001, P.R.China\\
  $^2\,$Department of Mathematics, Harbin Institute of Technology, Weihai, Shandong, 264209, P.R.China\\
   email\textup{\nocorr: \texttt{chenss@hit.edu.cn}}\\
}

\date{18 August 2021}
\pubyear{2000}
\volume{000}
\pagerange{\pageref{firstpage}--\pageref{lastpage}}

\begin{document}

\label{firstpage}
\maketitle

\begin{abstract}
In this paper, we consider a general single population model with delay and patch structure, which could model the population loss during the dispersal. It is shown that the model
admits a unique positive equilibrium when the dispersal rate is smaller than a critical value. The stability of the positive equilibrium and associated Hopf bifurcation are investigated when the dispersal rate is small or near the critical value. Moreover, we show the effect of network topology on Hopf bifurcation values for a delayed logistic population model.
\end{abstract}

\begin{keywords}
Hopf bifurcation; Patch structure; Delay; Population loss.
\end{keywords}

\begin{subjclass}[2010]
92D25, 34K18, 34K13, 37N25
\end{subjclass}

\section{Introduction}

The population dynamics can be investigated via reaction-diffusion systems or discrete patch models \cite{Akira2001,CantrellCosner}. For some biological species, time delays such as the maturation time and hunting time may have important
effect on the population dynamics, and it should be included in the modeling process. Therefore, various reaction-diffusion models with time
delay and delayed patch models
have been proposed to understand the interaction between biological species \cite{Magal2018,Wu1996Theory}.

For reaction-diffusion models with time
delay, time delay induced Hopf bifurcations and double Hopf bifurcations were studied extensively.
For example, one can refer to \cite{Faria2000,Gourley2002,Hadeler2007,Morita1984,ShiRuan2015,Yoshida1982The} and references therein for results on Hopf bifurcations of reaction-diffusion models with time
delay under the homogeneous Neumann boundary conditions, and
see \cite{DuNiu2019,DuNiu2020} for results on double Hopf bifurcations.
For the case of the homogeneous Dirichlet boundary conditions, delay induce Hopf bifurcations were studied in \cite{Busenberg1996,ChenShi2012JDE,ChenYu2016JDDE,ChenYu2016JDE,Guo2016,Hu2011,SuWeiShi2009,SuWeiShi2012,Yan2010} and references therein, and the bifurcating stable periodic solutions through Hopf bifurcation are usually spatially heterogeneous. Moreover, spatial heterogeneity was recently taken into consideration for reaction-diffusion models with time
delay, and the associated Hopf bifurcations
 were investigated in \cite{Chen2018Hopf,ChenWeiZhang2020,HuangChen2021,JinYuan2021,Li-Dai,ShiShi2019}.

There are also extensive results on bifurcations for delayed patch models. For the spatially homogeneous environments,
one can refer to \cite{ChangJin2020,DuanJin2020,Fernandes2019} and references therein for dispersal induced Turing bifurcations, and delay induced Hopf bifurcations were also studied extensively, see e.g. \cite{ChangSun2019,Madras1996,Petit2016,SoWuZou2001,Tian2019}.
Considering the spatial heterogeneity, Liao and Lou \cite{LiaoLou2014} investigated the following two-patch model, which models the growth of a single species:
\begin{equation}\label{patc}
\begin{cases}
\displaystyle \frac{d u_{1}}{d t}=d\left(\alpha_{11}u_{1}+\alpha_{12} u_{2}\right)+\mu u_{1}\left[m_{1}-u_{1}(t-r)\right], & t>0, \\
\displaystyle \frac{d u_{2}}{d t}=d\left(\alpha_{21}u_{1}+\alpha_{22} u_{2}\right)+\mu u_{2}\left[m_{2}-u_{2}(t-r)\right], & t>0,
\end{cases}
\end{equation}
where $u_j$ denotes
the population density in patch $j$ and time $t$, $d$ is the dispersal rate, $\mu$ is a scalar factor, $r$ represents the maturation time, and $m_j$ is the intrinsic growth
rate in patch $j$, which depends on patch $j$ and represents the spatial heterogeneity.
Dispersion matrix $A:=(\alpha_{jk})_{2\times 2}$ in
\cite{LiaoLou2014} is chosen to be
\begin{equation*}
(a) \;\alpha_{11}=\alpha_{22}=-1, \alpha_{12}=\alpha_{21}=1, \;\;\text{or}\;\;(b) \;\alpha_{11}=\alpha_{22}=-2, \alpha_{12}=\alpha_{21}=1,
\end{equation*}
where $\alpha_{jk}(j\ne k)\ge0$
denotes the rate of population movement from patch $k$ to patch $j$, and $\alpha_{jj}<0$ denotes the
rate of population leaving patch $j$.
Model \eqref{patc} with dispersion matrix $(a)$ (respectively, $(b)$) can be regarded as a discrete form of Hutchinson's model under the homogeneous Neumann (respectively, Dirichlet) boundary condition.
For case $(a)$, the dispersion matrix satisfies $-\alpha_{jj}=\sum_{k \neq j} \alpha_{kj}$ for $j=1,2$, which implies that the two patch habitat is closed, and there is no population loss during the dispersal. For case $(b)$, the dispersion matrix satisfies $-\alpha_{jj}>\sum_{k \neq j} \alpha_{kj}$, and the species has population loss at the boundary, see Fig. \ref{fig1}.
\begin{figure}[htbp]
\centering
\includegraphics[width=0.35\textwidth,trim=0 175 0 150,clip]{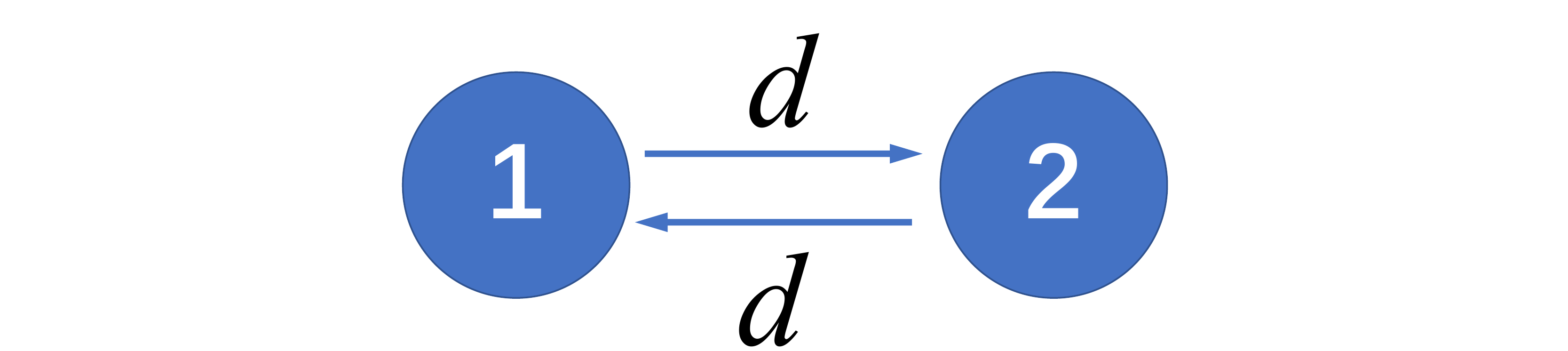}
\includegraphics[width=0.35\textwidth,trim=0 175 0 150,clip]{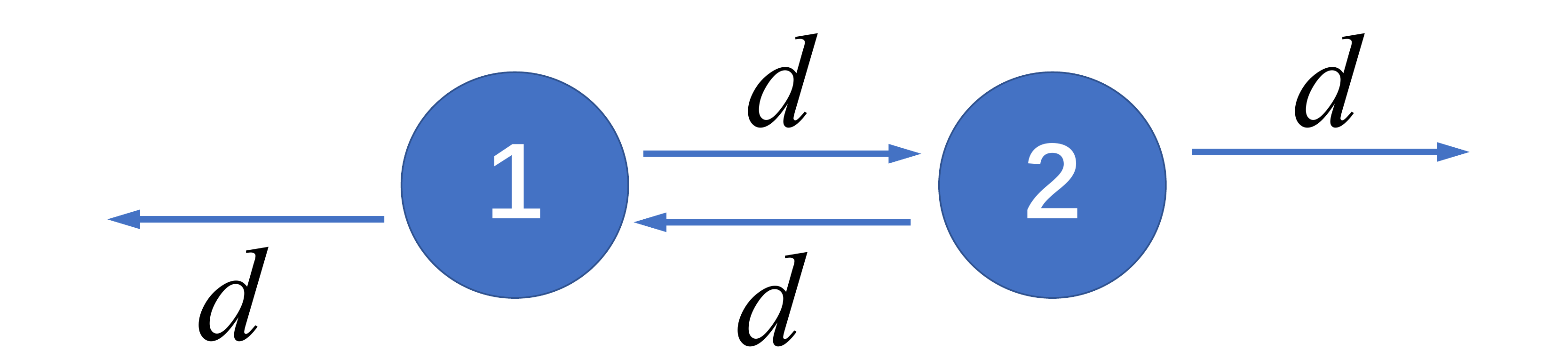}
\caption{The connection between two patches. (Left) Dispersion matrix $(a)$; (Right) Dispersion matrix $(b)$.
  \label{fig1}}
\end{figure}

A natural question is whether Hopf bifurcations can occur for model \eqref{patc} when the number of patches is finite but arbitrary, and in such a case, the connection among patches may also be complex. One can also refer to \cite{XiaoTang2011,ZhuYanJin} for detail discussions on complex connection among patches. In this paper, we aim to answer this question, and consider the following patch model:
\begin{equation}\label{M1}
\begin{cases}
\displaystyle \frac{d u_{j}}{d t}=d \sum_{k=1}^{n} \alpha_{jk} u_{k}+ u_j f_j \left(u_j, u_{j}(t-\tau)\right), & t>0,~~ j=1, \cdots, n, \\
\displaystyle \bm{u}(t)=\bm\psi(t) \geq \bm 0, & t \in[-\tau, 0].
\end{cases}
\end{equation}
Here $\bm u=\left(u_{1}, \cdots, u_{n}\right)^{T}$, where  $u_j$ stands for the number of individuals in patch $j$, $n \ge 2$ is the number of patches; $f_j(\cdot,\cdot)$ is the growth rate per capita; $d >0$ is the dispersal rate of the population; and time delay $\tau \ge 0$ represents the maturation time of the population.
Moreover, $A:=( \alpha_{jk})_{n \times n}$ is the dispersion matrix, where
$\alpha_{jk}(j \neq k)\ge 0$ denotes the rate of population movement from patch $k$ to patch $j$, and $\alpha_{jj}\le 0$ denotes the rate of population leaving patch $j$.

We remark that if there is no population loss during the dispersal ($-\alpha_{jj}=\sum_{k \neq j} \alpha_{kj}$ for $j=1,\dots,n$), Hopf bifurcation can occur when the dispersal rate is small, large or near some critical value, see \cite{ChenHopf,HuangChenS}.
Therefore, in this paper, we consider model \eqref{M1} when the species has population loss during the dispersal. That is, the following assumption holds:
\begin{enumerate}
\item [$\bf(H0)$] $A:=( \alpha_{jk})_{n \times n}$ is irreducible and essentially nonnegative; and $-\alpha_{jj}\ge\sum_{k \neq j} \alpha_{kj}$ for all $ j=1,\cdots,n$, and $-\alpha_{j j}>\sum_{k \neq j} \alpha_{kj}$ for some $j$.
\end{enumerate}
Here we remark that
real matrices with nonnegative off-diagonal elements are referred as essentially nonnegative matrices.
Throughout the paper, we also impose the following assumption:
\begin{enumerate}
\item [$\bf(H1)$] For $j=1,2,\cdots,n$, $f_j(x,y)\in C^4(\mathbb{R}\times \mathbb{R},\mathbb{R})$, $f_j(0,0)=m_j>0$ and $g_j'(x)<0$ for  $x>0$ with $g_j(x)=f_j(x,x)$.
\end{enumerate}
Here $m_j$ represents the intrinsic growth rate in patch $j$. The smooth condition that $f_j(x,y)\in C^4(\mathbb{R}\times \mathbb{R},\mathbb{R})$ is used to determine the direction of the Hopf bifurcation and the stability of the bifurcating periodic solutions, and we do not include this part in the paper for simplicity.
We remark that for the case of population loss, we need to modify the arguments in \cite{ChenHopf,HuangChenS} to derive \textit{a priori} estimates for eigenvalue problem. Moreover, we show the effect of dispersal rate $d$ and network topology on the Hopf bifurcation values for the logistic population model.




For simplicity, we give some notations here. For a matrix $D$, we denote the spectral bound of $D$ by
\begin{equation*}
s(D):=\max\{\mathcal {R}e \mu:\mu\text{ is an eigenvalue of } D\}.
\end{equation*}
For $\mu\in\mathbb{C}$, we denote the real and imaginary parts by $\mathcal {R}e\mu$ and $\mathcal {I}m \mu$, respectively.
For a space $Z$, we denote
complexification of $Z$ to be $Z_\mathbb{C}:= Z \oplus {\rm i}Z = \{x_1+{\rm i}x_2 | x_1, x_2 \in Z\}$. For a linear operator $T$, we define the domain and the kernel of $T$ by $\mathscr{D}(T)$ and $\mathscr{N}(T)$, respectively. For $\mathbb{C}^n$, we choose the inner product $\langle\bm{u}, \bm{v}\rangle=\sum_{j=1}^{n} \overline {u}_{j} v_{j}$ for $\bm{u}, \bm{v} \in \mathbb{C}^{n}$, and define the norm
\begin{equation*}
\|\bm u\|_{2}=\left(\sum_{j=1}^{n}\left|u_{j}\right|^{2}\right)^{1 / 2}.
\end{equation*}
For $\bm u=(u_1,\cdots, u_n)^T\in\mathbb R^n$, we  write $\bm u\gg \bm 0$ if $u_j>0$ for all $j=1,\cdots,n$.

The rest of the paper is organized as follows. In Section 2, we give some preliminaries, and show that model \eqref{M1}
admits a unique positive equilibrium  $\bm u_d$ for $d\in (0,d_*)$. In Section 3, we show the existence of the Hopf bifurcation when
$0<d\ll 1$ and $0<d_*-d\ll 1$, respectively.
In Section 4, we apply the obtained theoretical results to a logistic population model, discuss the effect of network topology on Hopf bifurcation values, and give some numerical simulations.

\section{Some preliminaries}
In this section, we cite some results on the properties of the spectrum bound $s\left(dA+\text{diag}(m_j)\right)$, and the global dynamics of model \eqref{M1} for $\tau=0$.
The first one is from \cite{chen2019spectral}.
\begin{lemma}\label{lem}
Assume that $\bf(H0)$ holds, and denote $s(d):=s\left(dA+\text{diag}(m_j)\right)$. Then $s(d)$ is strictly decreasing in $d\in(0,\infty)$, $\lim_{d\to0}s(d)=\max_{1\le j\le n}\{m_j\}$, and
$\lim_{d\to\infty} s(d)=-\infty$. Moreover, there exists $d_*>0$ such that $s(d_*)=0$, $s(d)>0$ for $d\in(0,d_*)$ and  $s(d)<0$ for $d>d_*$.
\end{lemma}
This, combined with \cite{chen2019spectral,li2010global,Lu1993,Zhao2017}, implies that:
\begin{lemma}\label{GAS}
Assume that $\bf(H0)$-$\bf(H1)$ hold, and $\tau=0$.
Then the trivial equilibrium $\bm 0= (0,\cdots,0)^T$ of \eqref{M1} is globally asymptotically stable for $d\ge d_*$, and for $d<d_*$, system \eqref{M1} admits a unique positive equilibrium
$\bm u^d=(u^d_{1},\cdots, u^d_{n})^T\gg\bm 0$, which is globally asymptotically stable.
\end{lemma}
It follows directly from the Perron-Frobenius theorem that $s\left(d_*A+\text{diag}(m_j)\right)(=0)$ is a simple eigenvalue of $d_*A+\text{diag}(m_j)$ with corresponding eigenvector $\bm\eta \gg \bm 0$ (or respectively, a simple eigenvalue of $d_* A^T+\text{diag}(m_j)$ with corresponding eigenvector $\bm\va\gg \bm 0$), where
\begin{equation}\label{bmeta}
\begin{split}
&\bm\eta=(\eta_1, \cdots, \eta_n)^T\;\;\text{where}\;\;\eta_j>0\;\;\text{for all}\;\;j=1,2,\cdots,n, \;\;\text{and}\;\; \sum_{j=1}^n {\eta_j}=1,\\
&\bm\va=(\va_1, \cdots, \va_n)^T,\;\;\text{where}\;\;\va_j>0 \;\;\text{for all}\;\;j=1,2,\cdots,n, \;\;\text{and}\;\; \sum_{j=1}^n {\va_j}=1.
\end{split}
\end{equation}
Then we have the following decomposition:
\begin{equation}\label{oplus}
\mathbb{R}^{n}=\operatorname{span}\{{\bm\eta}\} \oplus {X}_{1}=\operatorname{span}\{\bm \va\} \oplus \widetilde{X}_{1},
\end{equation}
where
\begin{equation}\label{X1}
\begin{split}
{X}_{1}&:=\left\{\bm x\in\mathbb R^n: \langle \bm \va,\bm x\rangle=0\right\}
=\left\{\left[d_*A+\text{diag}(m_j)\right]\bm y:\;\;\bm y\in \mathbb R^n\right\},\\
\widetilde{X}_{1}&:=\left\{\bm x\in\mathbb R^n: \langle \bm \eta,\bm x\rangle=0\right\}=\left\{\left[d_*A^T+\text{diag}(m_j)\right]\bm y: \;\;\bm y\in \mathbb R^n\right\}.
\end{split}
\end{equation}
To show the existence of Hopf bifurcation, we describe the profile of the unique positive equilibrium $\bm u^d$ as $d\to0$ or $d\to d_*$. Clearly, $\bm u^d=(u^d_{1},\cdots,u^d_{n})^T$ satisfies
\begin{equation}\label{steady states}
\displaystyle d\sum_{k=1}^{n} \alpha_{jk} u_{k}+ u_j f_j\left( u_{j}, u_{j}\right)=0,\;\; j=1, \cdots, n.
\end{equation}
\begin{lemma}\label{asymp}
Assume that $\bf(H0)$-$\bf(H1)$ hold.
Let $\bm u^d$ be the unique positive equilibrium of \eqref{M1} obtained in Lemma \ref{GAS} for $d\in(0,d_*)$, and denote
\begin{equation}\label{tilder1r2}
\tilde a:=\sum_{j=1}^{n}a_j \eta_j^2 \varsigma_j,\;\;\tilde b:=\sum_{j=1}^{n}b_j \eta_j^2 \varsigma_j,
\end{equation}
where $\bm\eta=(\eta_1, \cdots, \eta_n)^T$ and $\bm\va=(\va_1, \cdots, \va_n)^T$ are defined in \eqref{bmeta}, and
\begin{equation}\label{r1r2}
a_j:=\frac{\partial f_j(0,0)}{\partial x},\;\; b_j:=\frac{\partial f_j(0,0)}{\partial y}\;\;\text{for}\;\; j=1,2,\cdots,n.
\end{equation}
Then the following statements hold.
\begin{enumerate}
\item [$(i)$] Let $\bm u^d=(u_1^{0}, \cdots, u_n^{0})^T$ for $d=0$, where $u_j^{0}$ is the unique positive solution of $f_j(x,x)=0$ for $j=1,\cdots,n$.
Then $\bm u^d$ is continuously differentiable for $d\in[0,d_*)$.
\item [$(ii)$] There exists a continuously differentiable mapping $d \mapsto\left(\beta^{d}, \bm{\xi}^{d}\right)$ from $\left(0,d_{*}\right]$ to $\mathbb{R}^{+} \times X_{1}$ such that, for any $d \in\left(0,d_{*}\right)$, the unique positive equilibrium of \eqref{M1} can be represented as the following form
\begin{equation}\label{ula}
\bm {u}^d=\beta^{d}\left(d_{*}-d\right)\left[\bm{\eta}+\left(d_{*}-d\right) \bm {\xi}^{d}\right]. 
\end{equation}
Moreover,
\begin{equation}\label{beta*}
\beta^{d_{*}}=\frac{\sum_{j=1}^{n} m_j \eta_{j} \va_j}{-d_{*} (\tilde a+\tilde b)}>0,
\end{equation}
and $\bm{\xi}^{d_{*}}=\left(\xi^{d_{*}}_1, \cdots, \xi^{d_{*}}_n\right)^{T} \in X_1$ is the unique solution of the following equation
\begin{equation}\label{xi*}
d_*\left(d_*\sum_{k=1}^{n} \alpha_{jk} \xi_{k}+ m_j \xi_{j}\right)+{\eta}_{j}\left[m_j+ d_{*} \beta^{d_{*}}(a_j+b_j) {\eta}_{j} \right]=0, \;\; j=1, \cdots, n.
\end{equation}
\end{enumerate}
\end{lemma}
\begin{proof}
We first prove $(i)$.
It follows from assumption $\bf(H1)$ that $f_j(x,x)=0$ admits a unique positive solution, denoted by $u_j^{0}$.
Define \begin{equation*}
\begin{split}
&\bm G(d,\bm u)=
\left(\begin{array}{c}
d\sum_{k=1}^n\alpha_{1k}u_k+u_1 f_1(u_1,u_1)
\\
d\sum_{k=1}^n\alpha_{2k}u_k+u_2 f_2(u_2,u_2) \\
\vdots\\
d\sum_{k=1}^n\alpha_{nk}u_k+u_n f_n(u_n,u_n) \\
\end{array}\right).
\end{split}
\end{equation*}
Clearly, $\bm G(0,\bm u^0)=\bm 0$ and  $D_{\bm {u}} \bm G(0,\bm {u}^{0})=\operatorname{diag}\left(u_j^0 (a_j^0+b_j^0)\right)$, where
$D_{\bm {u}} \bm G(0,\bm {u}^{0})$ is the Fr\'echet derivative of $\bm G(d,\bm u)$ with respect to $\bm u$ at $(0,\bm {u}^{0})$, and
\begin{equation}\label{d12}
a_{j}^{0}=\left.\frac{\partial f_j}{\partial x}\right|_{\left( u_{j}^0,u_{j}^0\right)},\;\; b_{j}^{0}=\left.\frac{\partial f_j}{\partial y}\right|_{\left( u_{j}^0,u_{j}^0\right)},\;\;j=1,\cdots,n.
\end{equation}
By assumption $\bf(H1)$, we see that
\begin{equation}\label{aibi0n}
a_j^0+b_j^0<0 \;\;\text{for all}\;\;j=1, \cdots, n,
\end{equation}
which implies that $D_{\bm {u}} \bm G(0,\bm {u}^{0})$ is invertible. It follows from the implicit function theorem that
there exist $ d_1>0$ and a continuously differentiable mapping
$$d\in[0, d_1]\mapsto {{\bm u}}(d)=({u}_1(d),\cdots,{u}_n(d))^T\gg\bm 0$$
such that $\bm G(d,{{\bm u}}(d))=\bm 0$ and $ {{\bm u}}(0)=\bm u^0$. Therefore, $\bm u^d=\bm u(d)$, and $\bm u^d$ is continuously differentiable for $d\in [0,d_1]$. Note that $G(d,\bm u^d)=\bm 0$ for $d\in(0,d_*)$, and $\bm u^d$ is stable. Then, by the implicit function theorem, we obtain that $\bm u^d$ is continuously differentiable for $d\in (0,d_*)$. Here we omit the proof for simplicity.

Now, we prove $(ii)$. It follows from \eqref{oplus} that $\bm u^d$ can be represented as \eqref{ula}. Since $\bm u^d$ is continuously differentiable for $d\in (0,d_*)$, we see that  $\beta^d$ and $\bm {\xi}^{d}$ are also continuously differentiable for $d\in(0, d_*)$.
Then we will show that $\beta^d$ and $\bm {\xi}^{d}$ are continuously differentiable for $d=d_*$.

It follows from \eqref{aibi0n} that
\begin{equation}\label{tildeab<0n}
\tilde a+\tilde b <0,
\end{equation}
which implies that $\beta^{d_*}$ is positive. Since
$$\sum_{j=1}^n m_j {\eta}_{j} \va_j+ d_{*}\beta^{d_{*}} (\tilde a+\tilde b) =0,$$
we see that
$$\left({\eta}_{1}\left[m_1+ d_{*} \beta^{d_{*}}(a_1+b_1) {\eta}_{1} \right],\cdots,{\eta}_{n}\left[m_n+ d_{*} \beta^{d_{*}}(a_n+b_n) {\eta}_{n} \right]\right)^T\in X_1,$$
and consequently
$\bm{\xi}^{d_{*}}\in X_1$ is uniquely defined.

Multiplying \eqref{steady states} by $d_*$, we have
\begin{equation}\label{steady states2}
\displaystyle d\left[d_*\sum_{k=1}^{n} \alpha_{jk} u_{k}+m_ju_j\right]+(d_*-d)u_jm_j+ d_*u_j \left[f_j\left( u_{j}, u_{j}\right)-m_j\right]=0,\;\; j=1, \cdots, n.
\end{equation}
Substituting
$$\bm u=\beta(d_*-d)\left[{\bm \eta}+(d_*-d){\bm \xi}\right]$$
into \eqref{steady states2}, where $\bm \eta$ is defined in \eqref{bmeta} and $\bm \xi=(\xi_1,\cdots, \xi_n)^T\in X_1$, we see that $(\beta,\bm \xi)$ satisfies, for all $j=1,\cdots,n$,
\begin{equation*}
\begin{split}
p_j(d,\beta,\bm\xi)
:=d \left(d_*\sum_{k=1}^n \alpha_{jk}\xi_k+ m_j \xi_j\right) +[\eta_j+(d_*-d)\xi_j]\left(m_j+d_* q_j(d,\beta,\bm\xi)\right)=0,
\end{split}
\end{equation*}
where
\begin{equation}\label{qi}
q_{j}(d,\beta,\bm\xi)=\begin{cases}
\begin{array}{ll}
\displaystyle \frac{f_j \left(u_j, u_j\right)-m_j}{d_{*}-d}, & d \neq d_{*} \\
\displaystyle \beta (a_j+b_j) \eta_j, & d=d_{*}
\end{array}
\end{cases}
\end{equation}
with $u_j=\beta(d_*-d)\left[\eta_j+(d_*-d){\xi_j}\right]$.
Define $\bm p(d, \beta,\bm \xi):\mathbb R\times \mathbb R\times X_1\mapsto \mathbb R^n$ by
\begin{equation*}
\bm p(d, \beta,\bm \xi)=\left(p_1(d,\beta,\bm\xi),\cdots,p_n(d,\beta,\bm\xi)\right)^T.
\end{equation*}
Then $(d,\bm u)$ solves \eqref{steady states} if and only if $\bm p(d,\beta,\bm \xi)=\bm 0$ for $(\beta,\bm \xi)\in\mathbb R\times X_1$.
Clearly, $\bm p(d_*,\beta^{d_*},{\bm \xi}^{d_*})=\bm 0$, and  the Fr\'echet derivative of $\bm p$ with respect to $(\beta,\bm \xi)$ at $(d_*,\beta^{d_*},{\bm \xi}^{d_*})$ is
\begin{equation*}
D_{(\beta,\bm \xi)}\bm p(d_*,\beta^{d_*},{\bm \xi}^{d_*})[\epsilon,\bm v]=d_*\left(\begin{array}{c}
d_*\sum_{k=1}^n\alpha_{1k}v_k+m_1v_1+ (a_1+b_1) \eta_1^2\epsilon \\
d_*\sum_{k=1}^n\alpha_{2k}v_k+m_2v_2+ (a_2+b_2) \eta_2^2\epsilon \\
\vdots\\
d_*\sum_{k=1}^n\alpha_{nk}v_k+m_nv_n+ (a_n+b_n) \eta_n^2\epsilon \\
\end{array}\right),
\end{equation*}
where $\epsilon \in \mathbb R$ and $\bm v =(v_1,\cdots,v_n)^T\in X_1$. Since $\tilde a+\tilde b<0$ from \eqref{tildeab<0n},
we see that $D_{(\beta,\bm \xi)}\bm p(d_*,\beta^{d_*},{\bm \xi}^{d_*})$ is bijective from $\mathbb R\times X_1$ to $\mathbb R^n$.
It follows from the implicit function theorem that there
exist $d_1<d_*$ and a continuously differentiable mapping $d\in[d_1,d_*]\mapsto(\tilde \beta^d,\tilde {\bm \xi}^d)\in {\mathbb R}\times X_1$ such that
${\bm p}(d,\tilde \beta^d,\tilde {\bm \xi}^d)=\bm 0$, and $\tilde \beta^d=\beta^{d_*}$ and $\tilde{\bm \xi}^{d}={\bm \xi}^{d_*}$ for $d=d_*$.
The uniqueness of the positive equilibrium of \eqref{M1} implies that $\beta^d=\tilde \beta^d$ and $\bm {\xi}^{d}=\tilde{\bm \xi}^d$ for $d\in[d_1,d_*)$. Therefore, $\beta^d$ and $\bm {\xi}^{d}$ are continuously differentiable
for $d\in(0, d_*]$.
\end{proof}

\section{Stability and Hopf bifurcation}
In this section, we consider the stability of the unique positive equilibrium $\bm u^d$, and show the existence/nonexistence of a Hopf bifurcation for model \eqref{M1}. Linearizing \eqref{M1} at $\bm{u}^d$, we have
\begin{equation}
\label{2linear}
\displaystyle\frac{d \bm {v}}{d t} =dA\bm{ v}+{\rm diag}\left(f_j(u_j^d,u_j^d)\right)\bm v+{\rm diag}\left(u_j^d a_j^d\right)\bm v+ {\rm diag}(u_j^d b_j^d)\bm{ v}(t-\tau),
\end{equation}
where
\begin{equation}\label{abd}
a_{j}^{d}=\left.\frac{\partial f_j}{\partial x}\right|_{\left( u_{j}^d,u_{j}^d\right)},\;\; b_{j}^{d}=\left.\frac{\partial f_j}{\partial y}\right|_{\left( u_{j}^d,u_{j}^d\right)}.
\end{equation}
It follows from \cite{Wu1996Theory} that the solution semigroup of \eqref{2linear} has the infinitesimal generator $A_{\tau} (d)$ satisfying
\begin{equation*}
A_{\tau}(d) \bm{\Psi}=\dot{\bm{\Psi}},
\end{equation*}
and the domain of $A_{\tau}(d)$ is
\begin{equation*}
\begin{aligned}
\mathscr{D}\left(A_{\tau}(d)\right)=\big\{&\bm\Psi \in C_{\mathbb{C}} \cap C_{\mathbb{C}}^{1}:\bm\Psi(0)\in {\mathbb{C}}^n,
\dot{\bm\Psi}(0)=dA\bm \Psi(0)+\operatorname{diag} \left(f_j (u_j^d,u_j^d)\right)\bm \Psi(0) \\ &+\operatorname{diag}\left(u_j^d a_{j}^{d}\right)\bm\Psi(0)+\operatorname{diag}\left(u_j^d b_{j}^{d}\right)\bm \Psi(-\tau)\big\},
\end{aligned}
\end{equation*}
where $C_{\mathbb C}=C([-{\tau},0],\mathbb{C}^n)$ and $C^1_\mathbb{C}=C^1([-{\tau},0],\mathbb{C}^n)$. Then, we see that $\mu\in\mathbb{C}$ is an eigenvalue of $A_{\tau}(d)$, if and only if there exists $\bm{\varphi}=(\varphi_1,\cdots,\varphi_n)^T(\ne\bm 0)\in\mathbb{C}^n$ such that
\begin{equation}\label{2triangle}
\begin{split}
\Delta(d, \mu, \tau)\bm \varphi:=&dA\bm \varphi+ \operatorname{diag} \left(f_j (u^d_{j},u^d_{j})\right)\bm \varphi +  \operatorname{diag}\left(u^d_{j} a_{j}^{d}\right)\bm\varphi\\
&+ e^{-\mu \tau}\operatorname{diag}\left(u^d_{j} b_{j}^{d}\right)\bm\varphi - \mu \bm \varphi=\bm 0.
\end{split}
\end{equation}
Here the dispersion matrix $A$ may be asymmetric, and the environment can also be spatially heterogeneous. Therefore, one cannot obtain the explicit expression of $\bm u^d$. By Lemma \ref{asymp}, we obtain the asymptotic profile of $\bm u^{d}$ as $d\to0$ or $d\to d_*$.
Then the following discussion is divided into two cases: (I) $0<d_*-d\ll 1$, and (II) $0<d\ll 1$.

\subsection{The case of $0<d_*-d \ll 1$}
In this section, we will consider the existence of a Hopf bifurcation for \eqref{M1} with $0<d_*-d \ll 1$.
First, we obtain \textit{a priori} estimates for solutions of \eqref{2triangle}.
\begin{lemma}\label{bounded1}
Assume that $\left(\mu_{d}, \tau_{d}, \bm{\psi}_d\right)$ solves \eqref{2triangle} for $d \in (0,d_*)$, where $\mathcal{R}e \mu_{d}, \tau_{d} \ge 0$, and $\bm{\psi}_d=(\psi_{d,1},\cdots,\psi_{d,n})^T(\ne\bm 0)
\in\mathbb{C}^n$. Then there exists $d_1\in(0,d_*)$ such that $\left|\ds\frac{\mu_d}{d_*-d}\right|$ is bounded for $d\in[d_1,d_*)$. Moreover, ignoring a scalar factor, $\bm \psi_{{d}}$ can be represented as follows:
\begin{equation}\label{psirl1}
\begin{cases}
\bm \psi_{{d}}=r_{{d}} \bm \eta+ \bm w_{{d}},  \;\bm w_{{d}} \in\left(X_{1}\right)_{\mathbb{C}}, \; r_{{d}} \geq 0, \\
\|\bm{\psi}_d\|_2^2=\|\bm \eta\|_2^2,
\end{cases}
\end{equation}
where $\bm \eta$ is defined in \eqref{bmeta}, and $r_{{d}}$, $\bm w_{{d}}$ and $\bm \psi_{{d}}$ satisfy
$$\lim_{d \rightarrow d_*}r_{d}=1,\;\;\lim_{d \rightarrow d_*} \bm w_{d}=\bm 0,\;\;\lim_{d \rightarrow d_*}\bm \psi_{{d}}=\bm \eta. $$
\end{lemma}

\begin{proof}
We first show that $|\mu_d|$ is bounded for $d\in(0,d_*)$.
Substituting $\left(\mu_{d}, \tau_{d}, \bm{\psi}_d\right)$ into \eqref{2triangle}, we have
\begin{equation}\label{signlep}
\ds d\sum_{k=1}^n \alpha_{jk}\psi_{d,k}+f_j (u^d_{j},u^d_{j})\psi_{d,j}+u^d_{j}a_{j}^{d} \psi_{d,j}+u^d_{j} b_{j}^{d}\psi_{d,j}e^{-\mu_d\tau_d} -\mu_d\psi_{d,j}=0,\;\;j=1,\cdots,n.
\end{equation}
Multiplying \eqref{signlep} by $\overline \psi_{d,j}$ and summing the result over all $j$ yield
\begin{equation*}
\begin{aligned}
&d\sum_{j=1}^{n} \sum_{k =1}^n \alpha_{jk} \overline\psi_{d,j} \psi_{d,k}+ \sum_{j=1}^{n} f_j(u^d_{j},u^d_{j})|\psi_{d,j}|^2 +\sum_{j=1}^{n}u^d_{j} a_{j}^{d} |\psi_{d,j}|^2\\
&+ e^{-\mu_d \tau_d} \sum_{j=1}^{n}u^d_{j} b_{j}^{d} |\psi_{d,j}|^2- \mu_d \sum_{j=1}^{n}|\psi_{d, j}|^{2}=0.
\end{aligned}
\end{equation*}
Since $\|\bm{\psi}_d\|_2^2=\|\bm \eta\|_2^2$, we see that, for $d\in(0,d_*)$,
\begin{equation*}
\begin{aligned}
|\mu_d|\le& \max_{d\in[0,d_*],1\le j\le n}|f_j(u^d_{j},u^d_{j})|+ \max_{d\in[0,d_*],1\le j\le n}|u^d_{j}a_j^{d}| \\
&+ \max_{d\in[0,d_*],1\le j\le n}|u^d_{j}b_j^{d}|+n d_*\max_{1\le j,k\le n} |\alpha_{jk}|,
\end{aligned}
\end{equation*}
which implies that $|\mu_d|$ is bounded for $d\in(0,d_*)$.

Clearly, ignoring a scalar factor, $\bm \psi_{{d}}$ can be represented as \eqref{psirl1}.
Note from \eqref{psirl1} that $\|\bm{\psi}_d\|_2^2=\|\bm \eta\|_2^2$. Then, up to a subsequence, we can assume that
\begin{equation}\label{psistar111}
\lim_{d\to d_*}\mu_{d}=\gamma,\;\;\lim_{d\to d_*}\bm{\psi}_{d}=\lim_{d\to d_*}\left(r_{d}\bm \eta+ \bm w_{d}\right)=\bm \psi^*
\end{equation}
with $\mathcal{R}e \gamma\ge0$ and $\|\bm{\psi}^*\|_2^2=\|\bm \eta\|_2^2$.
This, combined with \eqref{2triangle}, implies that
\begin{equation*}
(d_*A+\text{diag}(m_j))\bm{\psi}^*-\gamma\bm{\psi}^*=\bm 0,
\end{equation*}
and consequently, $\gamma$ is an eigenvalue of $d_*A+\text{diag}(m_j)$.
Then, by \cite[Corollary 4.3.2]{Smith1995Monotone}, we have  $\gamma=s(d_*A+\text{diag}(m_j))=0$.
This, combined with \eqref{psirl1} and \eqref{psistar111}, implies that $\bm{\psi}^*=\bm \eta$, and consequently,
\begin{equation}\label{limrw21}
 \lim_{d\to d_*}r_{d}=1,\;\;\lim_{d\to d_* } \bm w_{d}=\bm 0.
\end{equation}

 Then multiplying \eqref{signlep} by $d_*$, we have
\begin{equation}\label{signlep2}
\begin{split}
&0=d\left[d_*\ds\sum_{k=1}^n \alpha_{jk}\psi_{d,k}+m_j\psi_{d,j}\right]+(d_*-d)m_j\psi_{d,j}+d_*\left[f_j (u^d_{j},u^d_{j})-m_j\right]\psi_{d,j}\\
&~~~~+d_*u^d_{j}a_{j}^{d} \psi_{d,j}+d_*u^d_{j} b_{j}^{d}\psi_{d,j}e^{-\mu_d\tau_d} -d_*\mu_d\psi_{d,j},\;\;\;\;\;\;\;j=1,\cdots,n.
\end{split}
\end{equation}
Plugging \eqref{ula} and \eqref{psirl1} into \eqref{signlep2}, we have, for $j=1,\cdots,n$,
\begin{equation}\label{Arla111}
\begin{aligned}
&0=\ds d\left(d_*\sum_{k=1}^n \alpha_{jk}w_{{d},k}+ m_j w_{{d},j} \right)+(d_*-d) \left[m_j+d_* q_j(d,\beta^d,\bm\xi^d)\right](r_{d} \eta_j+w_{{d},j})\\
&~~~~-d_* \mu_{d}(r_{d} \eta_j+ w_{{d},j})+d_*(d_*-d) \beta^{d}(a_j^{d}+b_j^{d}e^{-\mu_{d} \tau_{d}})\left[\eta_j+(d_*-d)\xi^{d}_j\right](r_{d} \eta_j+w_{{d},j}),
\end{aligned}
\end{equation}
where $q_j(d,\beta,\bm\xi)$ is defined in \eqref{qi}.
Note that $\bm \va$  is the eigenvector of $d_*A^T+\text{diag}(m_j)$ with respect to eigenvalue $0$. This, combined with \eqref{X1}, implies that
$$\sum_{j=1}^n\left(d_*\sum_{k=1}^n \alpha_{jk}w_{{d},k}+ m_j w_{{d},j} \right)\va_{j}=0,\;\;\sum_{j=1}^n w_{{d},j}\va_{j}=0.$$
Then multiplying \eqref{Arla111} by $\va_{j}$ and summing the result over all $j$ yield
\begin{equation}\label{H1}
\begin{split}
\ds\frac{\mu_{d}}{d_*-d}=&\frac{ \beta^{d}\sum_{j=1}^n \va_j (a_j^{d}+b_j^{d}e^{-\mu_{d} \tau_{d}})(r_{d} \eta_j+w_{{d},j})[\eta_j+(d_*-d) \xi^{d}_j]}{r_{d} \sum_{j=1}^n \eta_j \va_j}\\
&+\frac{\sum_{j=1}^n \va_j\left[m_j+{d_*} q_j(d,\beta^d,\bm\xi^d)\right](r_{d} \eta_j+ w_{{d},j})}{d_* r_{d} \sum_{j=1}^n \eta_j \va_j}.
\end{split}
\end{equation}
This, combined with \eqref{limrw21}, implies that there exists $d_1\in(0,d_*)$ such that $\left|\ds\frac{\mu_d}{d_*-d}\right|$ is bounded for $d\in[d_1,d_*)$.
\end{proof}

By Lemma \ref{bounded1}, we have the following result.
\begin{theorem}\label{sstability}
Assume that $\bf(H0)$-$\bf(H1)$ hold, and $\tilde a-\tilde b<0$,
where $\tilde a$ and $\tilde b$ are defined in \eqref{tilder1r2}.
Then there exists $d_2\in [d_1,d_*)$, such that
$$\sigma\left(A_{\tau}(d)\right) \subset\{x+ {\rm i} y: x, y \in \mathbb{R}, x<0\}\;\;\text{for}\;\;d\in[d_2,d_*)\;\;\text{and}\;\;\tau\ge0.$$
\end{theorem}

\begin{proof}
If the conclusion is not true, then there exists a positive sequence $\left\{d_{l}\right\}_{l=1}^{\infty}$ such that $\lim _{l \rightarrow \infty} d_{l}=d_*,$ and, for $l \geq 1$, $\Delta\left(d_{l}, \mu, \tau\right)\bm \psi=0$ is solvable for some value of
$\left(\mu_{{d_l}}, \tau_{{d_l}},\bm \psi_{{d_l}}\right)$ with $\mathcal{R} e \mu_{{d_l}}, \mathcal{I} m \mu_{{d_l}} \geq 0, \tau_{{d_l}} \geq 0$ and $\bm 0 \neq \bm \psi_{{d_l}} \in {\mathbb{C}}^n$. Note from the proof of Lemma \ref{bounded1} that $\left\{\left|\frac{\mu_{d_l}}{d_*-d_l}\right|\right\}_{l=1}^\infty$ and $\left\{\left| \mu_{d_l} \right|\right\}_{l=1}^\infty$ are bounded. Then we see that there exists a subsequence $\left\{d_{l_{k}}\right\}_{k=1}^{\infty}$ (we still use $\{{d}_l\}_{l=1}^{\infty}$ for convenience) such that
\begin{equation}\label{limtsig}
\lim_{l\to\infty}\frac{\mu_{d_l}}{d_*-d}=\mu^*,\;\;\lim_{l\to\infty}(e^{{-\tau_{d_l}}(\mathcal{R} e \mu_{{d_l}})}, e^{{-{\rm i}\tau_{d_l}}(\mathcal{I} m \mu_{{d_l}})})=(\sigma^*,e^{-{\rm i}\theta^*}),
\end{equation}
where \begin{equation*}
\sigma^*\in[0,1],\;\;\theta^*\in [0,2\pi),\;\;\mu^* \in \mathbb{C}\left(\mathcal{R} e \mu^{*},\mathcal{I} m \mu^{*} \geq 0\right).
\end{equation*}

It follows from Lemma \ref{bounded1} that
$\lim_{l \rightarrow \infty}r_{d_l}=1$, $\lim_{l \rightarrow \infty}\bm w_{d_l}=\bm 0$. By \eqref{ula} and \eqref{abd}, we have
$a_j^d=a_j$ and $b_j^d=b_j$ for $d=d_*$, where $a_j$ and $b_j$ are defined in \eqref{r1r2}. Then,
substituting $d=d_l$, $\mu_d=\mu_{d_l}$, $r_d=r_{d_l}$ and $\bm w_{d}=\bm w_{d_l}$ into \eqref{H1} and taking $l\to\infty$, we see from
\eqref{qi} and \eqref{limtsig} that
\begin{equation}\label{muesti}
\begin{split}
\ds\mu^*=
\frac{\sum_{j=1}^n\eta_j\va_j\big\{[m_j+d_* \beta^{d_*}\eta_j (a_j+b_j)]+d_* \beta^{d_*}\eta_j(a_j +b_j \sigma^* e^{-{\rm i} \theta^*})\big\} }{d_*\sum_{j=1}^n \eta_j \va_j}.
\end{split}
\end{equation}
By \eqref{beta*}, we have
$$\sum_{j=1}^n\eta_j\va_j\big\{[m_j+d_* \beta^{d_*}\eta_j (a_j+b_j)]=0.$$
This, combined with \eqref{tilder1r2} and \eqref{muesti}, yields
\begin{equation}\label{Eqconsin}
\begin{cases}
\ds \beta^{d_*}(\tilde a+\sigma^* \tilde b \cos\theta^* )=\mathcal{R} e \mu^{*}\sum_{j=1}^n\eta_j\va_j\ge0,\\
\mathcal{I}m \mu^{*}\sum_{j=1}^n\eta_j\va_j+ \beta^{d_*}\sigma^* \tilde b \sin\theta^* =0.
\end{cases}
\end{equation}
It follows from $\bf(H1)$ (see also \eqref{tildeab<0n}) that $\tilde a+\tilde b<0$.
Then if $\tilde a-\tilde b<0$, we have
$$
\tilde a<\min \left\{\tilde b,-\tilde b\right\} \leq 0 \text { and }-1<-\frac{\tilde b}{\tilde a}<1.
$$
This, combined with the first equation of \eqref{Eqconsin}, yields
$$
-\frac{\tilde b}{\tilde a}  \sigma^{*}\cos \theta^{*} \geq 1,$$
which is a contradiction. This completes the proof.
\end{proof}

From Theorem \ref{sstability}, we see that if $\tilde a -\tilde b<0$, then the positive equilibrium $\bm u^d$ is locally asymptotically stable for $0<d_*-d\ll 1$, and Hopf bifurcations can not occur.
Next, we show the existence of a Hopf bifurcation for $\tilde a -\tilde b>0$.
Clearly, $A_\tau(d)$ has a purely imaginary eigenvalue $\mu = {\rm i} \nu (\nu >0)$ for some $\tau \ge 0$, if and only if
\begin{equation}\label{Deltaiv}
\begin{split}
\bm H(d,\nu,\theta, \bm\varphi):=&dA\bm{\varphi}+{\rm diag}\left(f_j(u_j^d,u_j^d)\right)\bm \varphi+\operatorname{diag}\left(u_j^d a_{j}^{d}\right)\bm \varphi\\
&+e^{-{\rm i}\theta}{\rm diag}(u_j^d b_{j}^{d})\bm{ \varphi}-{\rm i}\nu\bm{\varphi}=\bm 0
\end{split}
\end{equation}
is solvable for some value of $ \nu>0, \theta \in [0, 2\pi)$ and $\bm \varphi (\ne\bm 0)\in \mathbb{C}^n$.
Ignoring a scalar factor, $\bm \psi (\ne\bm 0)\in \mathbb{C}^n$ in \eqref{Deltaiv} can be represented as follows:
\begin{equation}\label{psi}
\begin{split}
&\bm \psi=r \bm \eta+ \bm w,  \;\; \bm w \in\left(X_{1}\right)_{\mathbb{C}}, \;\; r \geq 0, \\
&\|\bm \psi\|_2^{2}=r^2\|\bm \eta\|_2^2+r \sum_{j=1}^n\eta_j(w_j+\overline w_j)+ \|\bm w\|_2^2=\|\bm \eta\|_2^2.
\end{split}
\end{equation}
Then we obtain an equivalent problem of \eqref{Deltaiv} as follows.

\begin{lemma}
Assume that $d\in(0,d_*)$. Then $(\nu,\theta,\bm \psi)$ is a solution of \eqref{Deltaiv}, where $\nu=(d_*-d)h>0$, $\theta\in[0,2\pi)$ and $\bm \psi$ satisfies \eqref{psi}, if and only if $(\bm w,r, h,\theta)$ solves the following system:
\begin{equation}\label{Fmain}
\begin{cases}
\bm F(\bm w,r, h,\theta, d)=(F_{1,1}, \cdots, F_{1,n},F_2,F_3)^T=\bm  0,\\
\bm w \in (X_1)_{\mathbb C},\;r\ge 0,\;h > 0, \;\theta \in[0,2 \pi).
\end{cases}
\end{equation}
Here
$\bm F(\bm w,r, h,\theta, d):(X_1)_{\mathbb C}\times \mathbb{R}^4 \mapsto (X_1)_{\mathbb C}\times \mathbb{C}\times \mathbb{R}$ is continuously differentiable, and
\begin{equation}\label{F}
\begin{cases}
\begin{aligned}
&\displaystyle F_{1,j}(\bm w,r, h,\theta, d):= d \left(d_*\sum_{k=1}^n \alpha_{jk}w_{k}+ m_j w_{j}\right)-(d_*-d)F_2(\bm w,r,h,\theta, d)\\
&~~~~~~~~~~~~~~~~~~~~~~~+(d_*-d)\left[m_j+d_*q_j(d,\beta^d,\bm\xi^d)-{\rm i}d_*h\right](r \eta_j+ w_{j})\\
&~~~~~~~~~~~~~~~~~~~~~~~+d_* (d_*-d)\beta^{d}\left[\eta_j+(d_*-d)\xi^d_j\right]\left(a_j^{d}+b_j^{d}e^{-{\rm i}\theta}\right)(r \eta_j+ w_{j}),\\
&\displaystyle F_2(\bm w,r,h,\theta, d):=\sum_{j=1}^n\va_j\left[m_j+d_*q_j(d,\beta^d,\bm\xi^d)-{\rm i}d_*h\right](r \eta_j+ w_{j})\\
&~~~~~~~~~~~~~~~~~~~~~~~+\sum_{j=1}^n\va_jd_*\beta^{d}[\eta_j+(d_*-d)\xi^{d}_j]\left(a_j^{d}+b_j^{d}e^{-{\rm i}\theta}\right)(r \eta_j+ w_{j}),\\
&\displaystyle F_3(\bm w,r,h,\theta, d):=(r^2-1)\|\bm \eta\|_2^2+r \sum_{j=1}^n\eta_j(w_j+\overline w_j)+ \|\bm w\|_2^2,
\end{aligned}
\end{cases}
\end{equation}
where $q_j(d,\beta,\bm\xi)$ and $a_j^{d}$, $b_j^{d}$ are defined in \eqref{qi} and \eqref{abd}, respectively.
\end{lemma}

\begin{proof}
Multiplying \eqref{Deltaiv} by $d_*$, we have
\begin{equation}\label{signlep21}
\begin{split}
&0=d\left[d_*\ds\sum_{k=1}^n \alpha_{jk}\varphi_{k}+m_j\varphi_{j}\right]+(d_*-d)m_j\varphi_{j}+d_*\left[f_j (u^d_{j},u^d_{j})-m_j\right]\varphi_{j}\\
&~~~~+d_*u^d_{j}a_{j}^{d} \varphi_{j}+d_*u^d_{j} b_{j}^{d}\varphi_{j}e^{-{\rm i}\theta} -{\rm i}d_*\nu\varphi_{j},\;\;\;\;\;\;\;j=1,\cdots,n.
\end{split}
\end{equation}
Then plugging \eqref{ula},  the first equation of \eqref{psi}, and $\nu=(d_*-d)h$ into \eqref{signlep21}, we have
$\bm y=(y_1,\cdots,y_n)^T=\bm 0$, where
\begin{equation}\label{Arla1111}
\begin{aligned}
&y_j:=\ds d \left(d_*\sum_{k=1}^n \alpha_{jk}w_{k}+ m_j w_{j}\right)\\
&~~~~~~+(d_*-d)\left[m_j+d_*q_j(d,\beta^d,\bm\xi^d)-{\rm i}d_*h\right](r \eta_j+ w_{j})\\
&~~~~~~+d_* (d_*-d)\beta^{d}\left[\eta_j+(d_*-d)\xi^d_j\right](a_j^{d}+b_j^{d}e^{-{\rm i}\theta})(r \eta_j+ w_{j}).
\end{aligned}
\end{equation}
Since $$\mathbb C^n=\operatorname{span}\{{\bm \rho}\} \oplus \left({X}_{1}\right)_{\mathbb C}\;\;\text{with}\;\;\bm \rho=(1,\cdots,1)^T,$$
we see that $$ \bm y=(d_*-d)F_2(\bm w,r,h,\theta, d)\bm \rho+\left(F_{1,1}(\bm w,r, h,\theta, d),\cdots,F_{1,n}(\bm w,r, h,\theta, d)\right)^T.$$
Therefore, $\bm y=\bm0$ if and only if $F_2(\bm w,r,h,\theta, d)=0$ and $F_{1,j}(\bm w,r, h,\theta, d)=0$ for all $j=1,\cdots,n$. Clearly, the second equation of \eqref{psi} is equivalent to $F_{3}(\bm w,r, h,\theta, d)=0$. This completes the proof.
\end{proof}

We first show that $\bm F(\bm w,r,h,\theta,d)=\bm 0$ has a unique solution for $d=d_*$.

\begin{lemma}\label{bmF0}
Assume that $\bf(H0)$-$\bf(H1)$ hold, and $\tilde a-\tilde b>0$, where $\tilde a$ and $\tilde b$ are defined in \eqref{tilder1r2}.
Then the following equation
\begin{equation}
\left\{\begin{array}{l}
\displaystyle {\bm F(\bm w,r, h, \theta,d_*)=\bm 0} \\
\displaystyle {\bm w \in (X_1)_{\mathbb C},\;r\ge 0,\;h \geq 0, \;\theta \in[0,2 \pi]}
\end{array}\right.
\end{equation}
has a unique solution $\left(\bm w_{d_*},r_{d_*}, h_{d_*},\theta_{d_*}\right)$, where
\begin{equation}\label{psi0theta0nu0}
\begin{aligned}
\bm w_{d_*}=\bm 0,\;\; r_{d_*}=1,\;\;h_{d_*}=\frac{ \beta^{d_*}\sqrt{{\tilde b}^2-{\tilde a}^2}}{\sum_{j=1}^n \eta_j \va_j},\;\;\theta_{d_*}=\arccos \left(-\tilde a / \tilde b\right),
\end{aligned}
\end{equation}
and $\beta^{d_*}$ is defined in \eqref{beta*}.
\end{lemma}

\begin{proof}
Set $\bm F_1=(F_{1,1}, \cdots, F_{1,n})^T$, and $\bm F_1(\bm w,r,h, \theta,d_*)=\bm 0$ if and only if $\bm w=\bm w_{d_*}=\bm 0$. This, together with $F_3(\bm w,r,h,\theta,d_*)= 0$, implies $r=r_{d_*}=1$.
Note from \eqref{ula} and \eqref{abd} that
$a_j^d=a_j$ and $b_j^d=b_j$ for $d=d_*$, where $a_j$ and $b_j$ are defined in \eqref{r1r2}.
Then, substituting $\bm w=\bm w_{d_*}$ and $r=r_{d_*}$ into $F_2(\bm w,r,h,\theta,d_*)= 0$, we see from \eqref{tilder1r2} and \eqref{beta*} that
\begin{equation}\label{F20}
d_* \beta^{d_*} \left(\tilde a+\tilde b e^{-{\rm i}\theta} \right) -{\rm i}d_*h \sum_{j=1}^n \eta_j \va_j=0,
\end{equation}
which implies that
\begin{equation}\label{sed}
\begin{cases}
\tilde a +\tilde b \cos \theta=0, \\
 \beta^{d_*} \tilde b \sin \theta +h\sum_{j=1}^n \eta_j \va_j=0.
\end{cases}
\end{equation}
It follows from $\bf(H1)$ (see also \eqref{tildeab<0n}) that $\tilde a+\tilde b<0$. Then if
 $\tilde a-\tilde b>0$, we have
\begin{equation}\label{tildeabab}
\tilde b<\min \left\{\tilde a,-\tilde a\right\} \leq 0 \text { and }-1<-\tilde a/ \tilde b<1.
\end{equation}
This, combined with \eqref{sed}, yields
\begin{equation*}
\begin{aligned}
\theta=\theta_{d_*}=\arccos \left(-\tilde a/ \tilde b\right),\;\;h=h_{d_*}=\frac{ \beta^{d_*}\sqrt{{\tilde b}^2-{\tilde a}^2}}{\sum_{j=1}^n \eta_j \va_j}.
\end{aligned}
\end{equation*}
This completes the proof.
\end{proof}

Then we solve $\bm F(\bm w,r,h,\theta,d)= \bm 0$ for  $0<d_*-d\ll 1$.

\begin{theorem}\label{uniquesolution}
Assume that $\bf(H0)$-$\bf(H1)$ hold, and $\tilde a-\tilde b>0$, where $\tilde a$ and $\tilde b$ are defined in \eqref{tilder1r2}. Then there exists $\tilde d_2$ $(0<d_*-\tilde d_2\ll 1)$ and a continuously differentiable mapping $d \mapsto (\bm w_{d}, r_d,h_{d},\theta_{d})$ from $[\tilde d_2, d_*]$ to $ (X_1)_{\mathbb C}  \times \mathbb{R}^3$ such that  $(\bm w_{d},r_d, h_{d},\theta_{d})$ is the unique solution of the following problem
\begin{equation}\label{Fla=0}
\begin{cases}
\bm F(\bm w,r, h, \theta,d)=\bm 0 \\
\bm w \in (X_1)_{\mathbb C},\;r\ge 0,\;h > 0, \;\theta \in[0,2 \pi)
\end{cases}
\end{equation}
for $d \in [\tilde d_2, d_*)$.
\end{theorem}

\begin{proof}
Let $\bm T(\bm \chi,\kappa, \epsilon,\vartheta)=(T_{1,1},\cdots,T_{1,n},T_2,T_3)^T: (X_1)_{\mathbb C} \times \mathbb{R}^3 \mapsto (X_1)_{\mathbb C} \times \mathbb{C} \times \mathbb{R}$ be the Fr\'{e}chet derivative of $\bm F(\bm w,r,h,\theta,d)$ with respect to $(\bm w,r, h, \theta)$ at $(\bm w_{d_*}, r_{d_*}, h_{d_*}, \theta_{d_*},{d_*}) $. A direct computation yields
\begin{equation*}
\begin{split}
T_{1j}(\bm \chi,\kappa, \epsilon,\vartheta)=&d_*\left(d_*\sum_{k=1}^{n} \alpha_{jk} \chi_{k}+ m_j\chi_{j}\right),\;\;j=1,\cdots,n, \\
T_{2}(\bm \chi,\kappa, \epsilon,\vartheta)=&\sum_{j=1}^n\va_j(\kappa\eta_j+\chi_j) \left\{m_j+d_*\beta^{d_*}(a_j+b_j)\eta_j +d_*\beta^{d_*}(a_j +b_j e^{-{\rm i}\theta_{d_*}})\eta_j-{\rm i}d_*h_{d_*}\right\} \\
&-{\rm i}\epsilon d_*\sum_{j=1}^n\va_j\eta_j-{\rm i}\vartheta d_*\beta^{d_*} \tilde b e^{-{\rm i}\theta_{d_*}}, \\
T_{3}(\bm \chi,\kappa, \epsilon,\vartheta)=&\sum_{j=1}^{n}\eta_j(\chi_j+\overline \chi_j)+2\kappa\|\bm \eta\|_2^2,
\end{split}
\end{equation*}
where we have used \eqref{tilder1r2} and \eqref{qi} to obtain $T_{2}$.

Now, we show that $\bm T$ is a bijection, and only need to show that $\bm T$ is an injective mapping. By \eqref{bmeta}-\eqref{X1}, we see that $d_*A+\text{diag}{(m_j)}$ is a bijection from $(X_1)_{\mathbb C}$ to $(X_1)_{\mathbb C}$. Then if $ T_{1j}(\bm \chi,\kappa, \epsilon,\vartheta)=\bm 0$ for all $j=1,\cdots,n$, we have $\bm \chi= \bm 0$. Substituting $\bm \chi=\bm 0$ into $T_{3}(\bm \chi,\kappa, \epsilon,\vartheta)=0$, we have $\kappa=0$. Then plugging $\bm \chi=\bm 0$ and $\kappa=0 $ into  $\bm T_{2}(\bm \chi,\kappa, \epsilon,\vartheta)=\bm 0$, we see from \eqref{psi0theta0nu0} that $\epsilon=\vartheta=0$. Therefore, $\bm T$ is an injection.
It follows from the implicit function theorem that there exists $\tilde d_2 \in [d_2, d_*)$ and a continuously differentiable mapping $d \mapsto(\bm w_{d},r_d, h_{d}, \theta_{d})$ from $[\tilde d_2, d_*]$ to $(X_1)_{\mathbb C} \times \mathbb{R}^3$ such that $(\bm w_{d},r_d, h_{d}, \theta_{d})$ satisfies \eqref{Fla=0}.

Then we prove the uniqueness of the solution of \eqref{Fla=0}. Actually, we only need to verify that if $(\bm w^{d},r^d, h^{d}, \theta^{d})$ satisfies \eqref{Fla=0}, then
$\displaystyle \left(\bm w^{d},r^d, h^{d},\theta^{d}\right) \rightarrow \left(\bm w_{d_*}, r_{d_*},h_{d_*},\theta_{d_*}\right)=\left(\bm 0,1,h_{d_*},\theta_{d_*}\right)$ as $d \rightarrow d_*.$
It follows from Lemma \ref{bounded1} that $h^d$ is bounded for $d\in[\tilde d _2, {d_*})$. Then, up to a subsequence, we can assume that $\lim_{d\to d_*}\theta^{d}=\theta^{d_*}$ and $\lim_{d\to d_*}h^{d}=h^{d_*}$.
It follows from Lemma \ref{bounded1} that
$\lim_{d\to d_*}r^{d}=r_{d_*}=1,$ $\lim_{d\to d_*}\bm w^{d}=\bm w_{d_*}=\bm 0.$
Taking the limits of $\bm F(\bm w^{d},r^{d}, h^{d},\theta^{d}, d)=\bm 0$
as $d\to d_*$, we have
\begin{equation*}
\bm F(\bm w_{d_*},r_{d_*}, h^{d_*},\theta^{d_*}, d_*)=\bm 0.
\end{equation*}
This, combined with Lemma \ref{bmF0}, implies that $\theta^{d_*}=\theta_{d_*}$ and $h^{d_*}=h_{d_*}$,
Therefore, $\left(\bm w^{d},r^d, h^{d},\theta^{d}\right) \rightarrow \left(\bm w_{d_*}, r_{d_*},h_{d_*},\theta_{d_*}\right)$ as $d\rightarrow {d_*} $.
This completes the proof.
\end{proof}

By Theorem \ref{uniquesolution}, we obtain the following result.
\begin{theorem}\label{solutionvtaupsi}
Assume that $\bf(H0)$-$\bf(H1)$ hold, and $\tilde a-\tilde b>0$, where $\tilde a$ and $\tilde b$ are defined in \eqref{tilder1r2}. Then for each $d\in[\tilde d_2, d_*)$, where $0<d_*-\tilde d_2\ll 1$, the following equation
$$
\left\{\begin{array}{l}
{\Delta(d, {\rm i} \nu, \tau)\bm \psi=\bm 0} \\
{\nu >0,\; \tau \geq 0,\; \bm \psi(\neq \bm 0) \in \mathbb{C}^{n}}
\end{array}\right.
$$
has a solution $(\nu,\tau, \bm \psi)$, if and only if
\begin{equation}\label{taunupsi}
\nu = \nu_{d}=(d_*-d)h_{d},\;\bm \psi=c \bm\psi_{d},\; \tau=\tau_{d,l}=\frac{\theta_d+2 l \pi}{\nu_d},\;\; l=0,1,2, \cdots,
\end{equation}
where $\bm \psi_{d}=r_{d}\bm \eta+\bm w_{d}$, $c$ is a nonzero constant, and $\bm w_{d},r_d, \theta_{d}, h_{d}$ are defined in Theorem \ref{uniquesolution}.
\end{theorem}

For further application, we consider the adjoint eigenvalue problem of \eqref{2triangle}.
For $\bm \psi, \widetilde{\bm \psi} \in \mathbb C^n$, we have
\begin{equation*}
\langle\widetilde{\bm \psi}, \Delta(d,{\rm i}\nu_d , \tau_{d,l})\bm \psi\rangle=\langle\widetilde {\Delta}(d,{\rm i}\nu_d, \tau_{d,l}) \widetilde{\bm \psi}, \bm \psi\rangle,
\end{equation*}
where
\begin{equation}\label{tildedelta}
\begin{split}
\widetilde{\Delta}(d,{\rm i}\nu_{d}, \tau_{d,l}) \widetilde {\bm \psi}= &d A^T\widetilde{\bm \psi} + \operatorname{diag} \left(f_j (u^d_{j},u^d_{j})\right)\widetilde{\bm \psi} + \operatorname{diag}\left(u^d_{j} a_{j}^{d}\right)\widetilde{\bm \psi}\\
&+ \operatorname{diag}\left(u^d_{j} b_{j}^{d}\right)\widetilde{\bm \psi} e^{{\rm i}\nu_d \tau_{d,l}}+{\rm i}\nu_d \widetilde{\bm \psi}.
\end{split}
\end{equation}
Here $\widetilde{\Delta}(d, {\rm i}\nu_d, \tau_{d,l})$ is the conjugate transpose matrix of ${\Delta}(d,{\rm i}\nu_d, \tau_{d,l})$. Clearly, $0$ is also an eigenvalue of $\widetilde{\Delta}(d, {\rm i}\nu_d, \tau_{d,l})$.


\begin{proposition}\label{p}
Let $\widetilde{\bm \psi}_{d}$ be the corresponding eigenvector of $\widetilde{\Delta}\left(d, {\rm i} {\nu_d}, {\tau}_{d,l}\right)$ with respect to eigenvalue $0$.
Then, ignoring a scalar factor, $\widetilde{\bm \psi}_{d}$ can be represented as follows:
\begin{equation}\label{psirl1s}
\begin{cases}
\widetilde{\bm \psi}_{d}=\widetilde{r}_{d} \bm \va+ \widetilde{\bm w}_{d},  \;\widetilde{\bm w}_{d} \in\left(\widetilde{X}_{1}\right)_{\mathbb{C}}, \; \widetilde{r}_{d}\geq 0, \\
\|\widetilde{\bm \psi}_d\|_2^2=\|\bm \va\|_2^2,
\end{cases}
\end{equation}
 and satisfies
\begin{equation}\label{va111}
\lim_{d\to d_*} \widetilde{\bm \psi}_{d}=\bm \va,
\end{equation}
where $\bm \va$ is defined in \eqref{bmeta}.
\end{proposition}

\begin{proof}
It follows from \eqref{psirl1s} that $\widetilde{\bm \psi}_{d}$ is bounded. Then, up to a subsequence, we can assume that
$\lim_{d\to d_*}\widetilde{\bm \psi}_{d}=\widetilde{\bm \psi}^* $.
Substituting $\widetilde{\bm \psi}=\widetilde{\bm \psi}_{d}$ into \eqref{tildedelta}, and taking $d\to d_*$, we have
\begin{equation}\label{conj}
\left(d_* A^T+ \text{diag}(m_j) \right)\widetilde{\bm \psi}^*=\bm 0,
\end{equation}
Noticing that $(d_* A^T+ \text{diag}(m_j) )\bm \va= \bm 0$, we see from \eqref{psirl1s} and \eqref{conj}
that
$\widetilde{\bm \psi}^*=\bm \va$. This completes the proof.
\end{proof}

For simplicity, we will always assume $d \in [\tilde d_2 ,d_*)$ in the following Theorems \ref{simpleeigenvalue}-\ref{stable}, where $0<d_*-\tilde d_2\ll 1$.
Actually, ˜$\tilde d_2$ may be chosen bigger than the one in Theorem \ref{uniquesolution} since further perturbation arguments are used. Next, we show that ${\rm i}\nu_{d}$ (obtained in Theorem \ref{solutionvtaupsi}) is simple, and the transversality condition holds.

\begin{theorem}\label{simpleeigenvalue}
Assume that $\bf(H0)$-$\bf(H1)$ hold, $\tilde a-\tilde b>0$, and $d \in [\tilde d_2 ,d_*)$, where $0<d_*-\tilde d_2\ll 1$. Then 
$\mu={\rm i} \nu_{d}$ is a simple eigenvalue of $A_{\tau_{d,l}}(d)$ for $l=0,1,2, \cdots$.
\end{theorem}

\begin{proof}
It follows from Theorem \ref{solutionvtaupsi} that $\mathscr{N}\left[A_{\tau_{d,l}}(d)-{\rm i} \nu_{d}\right]=\operatorname{span}\left[e^{{\rm i} \nu_{d} \theta}\bm \psi_d \right]$, where $\theta\in[- \tau_{d,l},0]$ and $\bm \psi_{d}$ is defined in Theorem \ref{solutionvtaupsi}. Then, we show that
$$\mathscr{N}\left[A_{\tau_{d,l}}(d)-{\rm i} \nu_{d}\right]^{2}=\mathscr{N}\left[{A}_{\tau_{d,l}}(d)-{\rm i} \nu_{d}\right].$$
If $\bm \phi \in \mathscr{N}\left[A_{\tau_{d,l}}(d)-{\rm i} \nu_{d}\right]^2$, then
$$
\left[A_{\tau_{d,l}}(d)-{\rm i} \nu_{d}\right]\bm \phi \in \mathscr{N}\left[A_{\tau_{d,l}}(d)-{\rm i} \nu_{d}\right]=\operatorname{span}\left[e^{{\rm i} \nu_{d} \theta}\bm \psi_d \right],
$$
and consequently, there exists a constant $\gamma$ such that
$$
\left[A_{\tau_{d,l}}(d)-{\rm i} \nu_{d}\right]\bm \phi=\gamma e^{{\rm i} \nu_{d} \theta}\bm \psi_d,
$$
which yields
\begin{equation}\label{dotphi}
\begin{aligned}
\dot{\bm\phi}(\theta) &={\rm i} \nu_{d}\bm \phi(\theta)+\gamma e^{{\rm i} \nu_{d} \theta}\bm \psi_d, \quad \theta \in\left[-\tau_{d,l}, 0\right], \\
\dot{\bm\phi}(0)=&dA\bm \phi(0)+ \operatorname{diag} \left(f_j (u^d_{j},u^d_{j})\right)\bm \phi(0) + \operatorname{diag}\left(u^d_{j} a_{j}^{d}\right)\bm\phi(0)\\
&+ \operatorname{diag}\left(u^d_{j} b_{j}^{d}\right)\bm \phi(-\tau_{d,l}).
\end{aligned}
\end{equation}
By the first equation of Eq. \eqref{dotphi}, we obtain that
\begin{equation}\label{phi}
\begin{aligned} \bm\phi(\theta) &=\bm\phi(0) e^{{\rm i} \nu_{d} \theta}+\gamma \theta e^{{\rm i} \nu_{d} \theta}\bm \psi_d, \\
\dot{\bm\phi}(0) &={\rm i} \nu_{d}\bm \phi(0)+\gamma \bm \psi_d.
\end{aligned}
\end{equation}
This, together with the second equation of \eqref{dotphi}, yields
\begin{equation}\label{Deltaphi1}
\begin{aligned}
\Delta\left(d, {\rm i} \nu_d, \tau_{d,l}\right)\bm \phi(0)=&dA\bm \phi(0)+ \operatorname{diag} \left(f_j (u^d_{j},u^d_{j})\right)\bm \phi(0) + \operatorname{diag}\left(u^d_{j} a_{j}^{d}\right)\bm\phi(0)\\
&+ e^{{-\rm i} \theta_{d}}\operatorname{diag}\left(u^d_{j} b_{j}^{d}\right)\bm\phi(0) - {\rm i}\nu_{d} \bm \phi(0)\\
=& \gamma \left(\bm\psi_d+ \tau_{d,l} e^{{-\rm i} \theta_{d}}\operatorname{diag}(u^d_{j} b_{j}^{d})\bm \psi_d\right).
\end{aligned}
\end{equation}
Multiplying both sides of \eqref{Deltaphi1} by $(\overline{\widetilde\psi}_{d,1}, \cdots, \overline{\widetilde\psi}_{d,n})$ to the left, we have
$$
\begin{aligned}
0 &=\left\langle  \widetilde \Delta\left(d, {\rm i} \nu_{d},\tau_{d,l}\right)\widetilde{\bm \psi}_{d}, \bm \phi(0)\right\rangle=\left\langle\widetilde{\bm \psi}_d, \Delta\left({d}, {\rm i} \nu_{d}, \tau_{d,l}\right)\bm \phi(0)\right\rangle\\
& = \gamma\left(\sum_{j=1}^n\overline{\widetilde\psi}_{d,j}\psi_{d,j}+ \tau_{d,l} e^{{-\rm i} \theta_{d}} \sum_{j=1}^n u^d_{j}b_{j}^{d}\overline{\widetilde\psi}_{d,j}\psi_{d,j}\right).
\end{aligned}
$$
Define
\begin{equation}\label{Sn}
S_{l}(d):=\sum_{j=1}^n\overline{\widetilde{\psi}}_{d,j}\psi_{d,j}+ \tau_{d,l} e^{{-\rm i} \theta_d} \sum_{j=1}^n u^d_{j}b_j^{d}\overline{\widetilde{\psi}}_{d,j}\psi_{d,j}.
\end{equation}
By Theorems \ref{uniquesolution}, \ref{solutionvtaupsi} and \eqref{va111}, we have $\bm \psi_d\to \bm \eta$, $\widetilde{\bm \psi}_d\to \bm \va$, $\theta_d\to \theta_{d_*}$, $(d_*-d)\tau_{d,l}\to \frac{\theta_{d_*}+2l\pi}{h_{d_*}}$ and $b_j^{d}\to b_j$ for $j=1,\cdots,n$ as $d\to d_*$, where $\theta_{d_*}$ and $h_{d_*}$ are defined in \eqref{psi0theta0nu0}.
Then we see from \eqref{ula} and \eqref{psi0theta0nu0} that
\begin{equation*}
\lim_{d\to{d_*}}S_l(d)=\sum_{j=1}^n \va_j \eta_j \left[1+\left(\theta_{d_*}+2 l \pi\right)\left(\frac{-\tilde a}{\sqrt{\tilde{b}^{2}-\tilde a^{2}}} +{\rm i}\right)\right]
\neq 0,
\end{equation*}
which implies that $\gamma=0$ for $d \in [\tilde d_2, d_*)$, where $0<d_*-\tilde d_2\ll 1$. Therefore, for any $l=0,1,2,\cdots$,
$$
\mathscr{N}[A_{\tau_{d,l}}(d)-{\rm i}\nu_d]^j
=\mathscr{N}[A_{\tau_{d,l}}(d)-{\rm i}\nu_d],\;\;j=
2,3,\cdots,
$$
and consequently, ${\rm i}\nu_d$ is a simple eigenvalue of
$A_{\tau_{d,l}}(d)$ for $l=0,1,2,\cdots$.
\end{proof}

By Theorem \ref{simpleeigenvalue}, we see that $\mu=\textrm{i}\nu_{d}$ is a simple eigenvalue of $A_{\tau_{d,l}}(d)$. Then, it follows from the implicit function theorem, for each $l=0,1,\cdots$, there exists a neighborhood $O_{l}\times D_{l}\times H_{l}$ of $(\tau_{d,l},\textrm{i}\nu_d,{\bm \psi}_d)$ and a continuously
differentiable function $(\mu(\tau),\bm\psi(\tau)):O_{q,l}\rightarrow D_{q,l}\times H_{q,l}$ such that $
\mu(\tau_{d,l})=\textrm{i}\nu_d$, $\bm \psi(\tau_{d,l})={\bm \psi}_d$, and for each $\tau\in O_{l}$, the only eigenvalue of $A_\tau(d)$ in $D_{l}$ is $\mu(\tau),$ and
\begin{equation}\label{Deltapsi}
\begin{split}
\Delta(d,\mu(\tau),\tau)\bm\psi(\tau)=&dA\bm \psi(\tau)+ \operatorname{diag} \left(f_j (u^d_{j},u^d_{j})\right)\bm \psi(\tau) + \operatorname{diag}\left(u^d_{j} a_{j}^{d}\right)\bm\psi(\tau)\\
&+ e^{-\mu(\tau) \tau}\operatorname{diag}\left(u^d_{j} b_{j}^{d}\right)\bm\psi(\tau) - \mu(\tau) \bm \psi(\tau)=\bm 0.
\end{split}
\end{equation}
Then, we prove that the following transversality condition holds.

\begin{theorem}\label{Remu}
Assume that $\bf(H0)$-$\bf(H1)$ hold, $\tilde a-\tilde b>0$, and $d \in [\tilde d_2, d_*)$, where $0<d_*-\tilde d_2\ll 1$. Then
$$
\frac{d \mathcal{R} e\left[\mu\left(\tau_{d,l}\right)\right]}{d \tau}>0, \quad l=0,1,2, \cdots.
$$
\end{theorem}
\begin{proof}
Differentiating Eq. \eqref{Deltapsi} with respect to $\tau$ at $\tau=\tau_{d,l}$, we have
\begin{equation}\label{dmu}
\begin{split}
\displaystyle -\frac{d \mu\left(\tau_{d,l}\right)}{d \tau} &\left( \tau_{d,l}\operatorname{diag}(u^d_{j}b_{j}^{d})\bm\psi_{d}e^{-{\rm i}\theta_{d}}+\bm \psi_d \right)+\Delta\left(d, {\rm i} \nu_d, \tau_{d,l}\right) \frac{\bm \psi \left(\tau_{d,l}\right)}{d \tau}\\
&- {\rm i} \nu_d \operatorname{diag}(u^d_{j}b_{j}^{d})\bm\psi_{d}e^{-{\rm i}\theta_d} =\bm 0.
\end{split}
\end{equation}
Clearly,
\begin{equation*}
\left\langle \widetilde{\bm \psi}_d,\Delta\left(d, {\rm i} \nu_d, \tau_{d,l}\right) \frac{d \bm \psi\left(\tau_{d,l}\right)}{d \tau}\right\rangle=\left\langle \widetilde\Delta\left(d,{\rm i} \nu_d,\tau_{d,l}\right) \widetilde{\bm \psi}_d, \frac{d \bm \psi\left(\tau_{d,l}\right)}{d \tau}\right\rangle= 0.
\end{equation*}
Then, multiplying both sides of Eq. \eqref{dmu} by $(\overline{\widetilde\psi}_{d,1}, \cdots, \overline{\widetilde\psi}_{d,n})$ to the left, we have
\begin{equation*}
\begin{aligned}
\frac{d \mu\left(\tau_{d,l}\right)}{d \tau}=& \frac{-{\rm i} \nu_{d}\sum_{j=1}^{n} u^d_{j} b_{j}^{d}\overline{\widetilde\psi}_{d,j}\psi_{d,j}e^{-{\rm i}\theta_{d}}}{\sum_{j=1}^{n}\overline{\widetilde\psi}_{d,j}\psi_{d,j}+ \tau_{d,l} \sum_{j=1}^{n} u^d_{j} b_{j}^{d} \overline{\widetilde\psi}_{d,j}\psi_{d,j} e^{-{\rm i}\theta_{d}}} \\
=&\frac{1}{\left|S_{l}(d)\right|^{2}}\left[-{\rm i}  \nu_{d}e^{-{\rm i}\theta_{d}}\left(\sum_{j=1}^{n}{\widetilde\psi}_{d,j}\overline{\psi}_{d,j} \right) \sum_{j=1}^{n}  u^d_{j} b_{j}^{d} \overline{\widetilde\psi}_{d,j}\psi_{d,j}\right.\\
&\left.-{\rm i} \nu_{d}\tau_{d,l} \left(\sum_{j=1}^{n} u^d_{j} b_{j}^{d} \overline{\widetilde\psi}_{d,j}\psi_{d,j}\right) \left(\sum_{j=1}^{n} u^d_{j} b_{j}^{d} {\widetilde\psi}_{d,j}\overline{\psi}_{d,j}\right)\right].
\end{aligned}
\end{equation*}
It follows from Theorems \ref{uniquesolution}, \ref{solutionvtaupsi} and \eqref{va111} that $\bm \psi_{d}\to \bm \eta$, $\widetilde{\bm \psi}_{d}\to \bm \va$, $\theta_d\to \theta_{d_*}$, $\ds \frac{\nu_{d}}{d_*-d}= {h_{d}}\to {h_{d_*}}$ and $b_j^{d}\to b_j$ for $j=1,\cdots,n$ as $d\to d_*$, where $\theta_{d_*}$ and $h_{d_*}$ are defined in \eqref{psi0theta0nu0}.
Then we see that
$$\begin{aligned}
\lim _{d \rightarrow {d_*}} &\frac{1}{(d_*-d)^2}\frac{d \mathcal{R} e\left[\mu\left(\tau_{d,l}\right)\right]}{d \tau}=\frac{ \left(\beta^{d_*} \right)^2 \left({\tilde b}^2-{\tilde a}^2\right)}{\lim _{d\rightarrow d_*}\left|S_{l}(d)\right|^{2}}>0,
\end{aligned}$$
where we have used \eqref{tildeabab} in the last step.
This completes the proof.
\end{proof}

By Theorems \ref{solutionvtaupsi}, \ref{simpleeigenvalue} and \ref{Remu}, we obtain the main result for this subsection.
\begin{theorem}\label{stable}
Assume that $\bf(H0)$-$\bf(H1)$ hold and $d \in [\tilde d_2, d_*)$, where $0<d_*-\tilde d_2 \ll 1$. Let $\bm {u}^{d}$ be the positive equilibrium of model \eqref{M1} obtained in Lemma \ref{GAS}. Then the following statements hold.
\begin{enumerate}
\item [{\rm (i)}] If $\tilde a-\tilde b<0$, where $\tilde a$ and $\tilde b$ are defined in \eqref{tilder1r2},
then $\bm {u}^{d}$ is locally asymptotically stable for $\tau \in\left[0, \infty\right)$.
\item [{\rm (ii)}] If $\tilde a-\tilde b>0,$
then there exists $\tau_{d,0}>0$ such that $\bm {u}^{d}$ of \eqref{M1} is locally asymptotically stable for $\tau \in\left[0, \tau_{d,0}\right)$, and unstable for $\tau \in\left(\tau_{d,0}, \infty\right).$ Moreover, when $\tau=\tau_{d,0},$ system \eqref{M1} undergoes a Hopf bifurcation at $\bm {u}^{d}$.
\end{enumerate}
\end{theorem}

\subsection{The case of $0<d\ll 1$}
In this section, we will consider the case of $0<d\ll 1$.
First, we give \textit{a priori} estimates for solutions of \eqref{2triangle}.
\begin{lemma}\label{3nu}
Assume that $\left(\mu^d, {\tau}^d, \bm{\varphi}^d \right)$ solves \eqref{2triangle}, where $\mathcal{R}e \mu^{d}, {\tau}^d \ge 0$, and $\bm{\varphi}^d=(\varphi_1^d,\cdots,\varphi_n^d)^T(\ne\bm 0)
\in\mathbb{C}^n$. Then for any $\tilde d>0$, $\left|\mu^d\right|$ is bounded for $d\in(0,\tilde d]$.
\end{lemma}

\begin{proof}
Without loss of generality, we assume that $\|\bm{\varphi}^d\|_2^2=1$.
Substituting $(\mu^d, {\tau}^d, \bm{\varphi}^d)$ into \eqref{2triangle} and multiplying both sides of \eqref{2triangle} by $\left(\overline {\varphi_1^d},\cdots,\overline {\varphi_n^d}\right)$ to the left, we obtain that
\begin{equation*}
\begin{split}
\left(\overline {\varphi_1^d},\cdots,\overline
{\varphi_n^d}\right)&\left[dA\bm{\varphi}^d+{\rm diag}\left(f_j(u_j^d,u_j^d)\right)\bm \varphi^d+\operatorname{diag}\left(u_j^d a_{j}^{d}\right)\bm{ \varphi}^d\right.\\
&\left.+e^{-\mu^d {\tau}^d}{\rm diag}(u_j^d b_{j}^{d})\bm{ \varphi}^d-\mu^d\bm{\varphi}^d\right]=0.
\end{split}
\end{equation*}
Then, for $d\in(0,\tilde d]$, we have
\begin{equation*}
\left|\mu^d\right|\le \max_{d\in[0,\tilde d],1\le j\le n}|f_j(u_j^d,u_j^d)|+ \max_{d\in[0,\tilde d],1\le j\le n}|u_{j}^{d} a_{j}^{d}|+ \max_{d\in[0,\tilde d],1\le j\le n}|u_{j}^{d} b_{j}^{d}|+\tilde d n\max_{1\le j,k\le n} |\alpha_{jk}|,
\end{equation*}
and consequently, $\left|\mu^d\right|$ is bounded for $d\in(0,\tilde d]$.
\end{proof}

Using similar arguments as in the proof of Theorem \ref{sstability}, we can obtain the following result, and here we omit the proof for simplicity.

\begin{theorem}\label{sstability1}
Assume that $\bf(H0)$-$\bf(H1)$ hold, and
$a_j^0-b_j^0<0$ for all $j=1,\cdots,n$,
where $a_j^0$ and $b_j^0$ are defined in \eqref{d12}.
Then there exists $\hat d_1\in (0,d_*]$, such that
$$\sigma\left(A_{\tau}(d)\right) \subset\{x+ {\rm i} y: x, y \in \mathbb{R}, x<0\}\;\;\text{for}\;\;d\in(0,\hat d_1]\;\;\text{and}\;\;\tau\ge0.$$
\end{theorem}
It follows from Theorem \ref{sstability1} that if $a_j^0-b_j^0<0$ for all $j=1,\cdots,n$, then Hopf bifurcations can not occur for $0<d\ll 1$.
Then we define
\begin{equation}\label{M}
\mathcal {M}=\left\{j\in 1, \cdots, n:a_j^0-b_j^0>0\right\},
\end{equation}
and show that Hopf bifurcations can occur when $\mathcal {M} \ne \emptyset $.
For simplicity, we impose the following assumption:
\begin{enumerate}
\item [$\bf(H2)$] $a_j^0-b_j^0> 0$ for $j=1,\cdots,p$, and $a_j^0-b_j^0<0$ for $j=p+1,\cdots,n$, where $1\le p\le n$.
\end{enumerate}

In fact, if the patches are independent of each other ($d=0$), we have
\begin{equation}\label{Md0}
\displaystyle u_{j}'= u_j f_j \left(u_j, u_{j}(t-\tau)\right),\; t>0,\;\;j=1, \cdots, n.
\end{equation}
A direct computation implies the following result.
\begin{lemma}\label{pindependent}
Assume that $\bf(H1)$-$\bf(H2)$ hold. Then for each $1\le j\le n$, model \eqref{Md0} admits a unique positive equilibrium $u_j^0$, where $u_j^0$ (defined in Lemma \ref{asymp}) is the unique positive solution of $f_j(x,x)=0$.
Moreover, the following statements hold.
\begin{itemize}
  \item [$(i)$] For each $1\le j \le p$, the unique positive equilibrium $u_j^0$ of model \eqref{Md0} is locally asymptotically stable when $\tau\in[0, \tau_{j}^0)$, and
unstable when $\tau \in(\tau_{j}^0,\infty)$. Moreover, when $\tau=\tau_{j}^0$, model \eqref{Md0} undergoes a Hopf bifurcation, where
\begin{equation}\label{tauthetanu}
\ds\tau_{j}^0=\frac{\theta_j^0}{\nu_j^0}\;\;\text{with}\;\;\theta_j^0=\arccos(-a_{j}^{0}/b_{j}^{0})\in(0,\pi)\;\ \text{and}\;\; \nu_j^0=u_j^0\sqrt{(b_{j}^{0})^2-(a_{j}^{0})^2}>0.
\end{equation}
  \item [$(ii)$] For each $p+1\le j \le n$,  the unique positive equilibrium $u_j^0$ of model \eqref{Md0} is locally asymptotically stable for $\tau\ge0$.
\end{itemize}
\end{lemma}

Now, we consider the solution of \eqref{Deltaiv} for $d=0$.

\begin{lemma}\label{limz}
Assume that $\bf(H1)$-$\bf(H2)$ hold, $d=0$, and
\begin{equation}\label{isolp}
\ds\left({{\nu_j^0}, \theta_j^0}\right)\ne \left({{\nu_k^0},\theta_k^0}\right) \;\;\text{for any}\;\;j\ne k \;\;\text{and}\;\;1\le j,k\le p,
\end{equation}
where $\theta_j^0$ and $\nu_j^0$ are defined in \eqref{tauthetanu}
for $j=1,\cdots,p$.
Then
\begin{equation}
\left\{(\nu,\theta):\nu\ge0,\;\theta\in[0,2\pi], \;\mathcal S^0(\nu,\theta)\ne \{\bm 0\}\right\}=\left\{(\nu_q^0,\theta_q^0)\right\}_{q=1}^p,
\end{equation}
where $(\nu_q^0,\theta_q^0)\in(0,\infty)\times(0,\pi)$, and
\begin{equation*}
\mathcal S^0(\nu,\theta):=\{\bm \varphi: \bm H(0, \nu,\theta,\bm\varphi)=\bm 0\}
\end{equation*}
with $\bm H(d,\nu,\theta, \bm\varphi)$ defined in \eqref{Deltaiv}.
Moreover, denoting $\mathcal S_q=\mathcal S^0(\nu_q^0,\theta_q^0)$ for any $q=1,\cdots,p$, we have $\mathcal S_q=\{c\bm\varphi_q^0:c\in\mathbb C\}$, where
${\bm \varphi}^{0}_q=(\varphi^{0}_{q,1},\cdots,\varphi^{0}_{q,n})$,
$\varphi^{0}_{q,q}=1$ and $\varphi^{0}_{q,j}=0$ for $j\ne q$.
\end{lemma}
\begin{proof}
It follows from Lemma \ref{asymp} that  $u_j^{0}$ satisfies $f_j(u_j^{0},u_j^{0})=0$ for $j=1,\cdots,n$.
Therefore, if there exists $\bm \varphi\ne\bm 0$ such that
$\bm H(0, \nu,\theta,\bm\varphi)=\bm0$, then
\begin{equation*}
\prod_{i=1}^n(u_j^0a_{j}^{0}+u_j^0 b_{j}^{0}e^{-{\rm i}\theta}-{\rm i} \nu)=0,
\end{equation*}
and consequently, for $j=1,\cdots,n$,
\begin{equation*}
\begin{cases}
a_{j}^{0} +b_{j}^{0} \cos \theta=0, \\
u_j^0 b_{j}^{0} \sin \theta +\nu=0.
\end{cases}
\end{equation*}
It follows from $\bf(H1)$ and $\bf(H2)$ that $a_j^0+b_j^0<0$ for $j=1, \cdots, n$ and
$a_j^0-b_j^0>0$ for $j=1, \cdots, p$. Then, for $j=1,\cdots,p$,
\begin{equation}\label{sab}
b_{j}^{0}<\min \left\{a_{j}^{0},-a_{j}^{0}\right\} \leq 0 \text { and }-1<-a_{j}^{0}/ b_{j}^{0}<1,
\end{equation}
which leads to $\nu=\nu_q^0$, $\theta=\theta_q^0$ for $q=1,\cdots,p$, where $\nu_q^0$ and $\theta_q^0$ are defined in \eqref{tauthetanu}.
Since $\ds\left({{\nu_j^0}, \theta_j^0}\right)\ne \left({{\nu_k^0},\theta_k^0}\right)$ for any $j\ne k$ and $1\le j,k\le p$,
it follows that $\mathcal S^0_q=\{c\bm\varphi_q^0:c\in\mathbb C\}$. This completes the proof.
\end{proof}

\begin{remark}
We remark that if
\begin{equation}\label{strong}
\ds\frac{\theta_j^0}{\nu_j^0}\ne\frac{\theta_k^0}{\nu_k^0} \;\;\text{for any}\;\;j\ne k \;\;\text{and}\;\;1\le j,k\le p,
\end{equation}
then \eqref{isolp} in Lemma \ref{limz} hold. By Lemma \ref{pindependent}, we see that \eqref{strong} implies that the first Hopf bifurcation values of model  \eqref{Md0} for $1\le j\le p$ are not identical. That is, the first Hopf bifurcation values of each
isolated patch $j$ for $1\le j\le p$ are not identical.
\end{remark}
Then we consider the solution of \eqref{Deltaiv} for $0<d\ll 1$.
\begin{lemma}\label{l4.4}
Assume that $\bf(H0)$-$\bf(H2)$ and \eqref{isolp} hold,
and $d\in(0,\tilde d)$ with $0<\tilde d\ll 1$. Then there exists $p$ pairs of $(\nu_q^d,\theta_q^d)\in(0,\infty)\times (0,\pi)$ such that
\begin{equation}\label{4.8}
\left\{(\nu,\theta):\nu\ge0,\;\theta\in[0,2\pi), \;\mathcal S^d(\nu,\theta)\ne \{\bm 0\}\right\}=\left\{(\nu_q^d,\theta_q^d)\right\}_{q=1}^p,
\end{equation}
where
\begin{equation*}
\mathcal S^d(\nu,\theta):=\{\bm \varphi: \bm H(d, \nu,\theta,\bm\varphi)=\bm 0\}
\end{equation*}
with $\bm H(d,\nu,\theta, \bm\varphi)$ defined in \eqref{Deltaiv}.
Moreover, denoting $\mathcal S^d_q=S^d(\nu_q^d,\theta_q^d)$ for any $q=1,\cdots,p$, we have $\mathcal S^d_q=\{c\bm\varphi_q^d:c\in\mathbb C\}$, and
\begin{equation*}\label{limd}
\lim_{d\to 0} \nu_q^d= \nu_q^0=u_q^0\sqrt{(b_{q}^{0})^2-(a_{q}^{0})^2}, \;\;\lim_{d\to0} \theta_q^d= \theta_q^0=\arccos(-a_{q}^{0}/b_{q}^{0}) \;\;\text{and}\;\;\lim_{d\to0}\bm \varphi_q^d=\bm \varphi_q^0,
\end{equation*}
 where $\nu_q^0$, $\theta_q^0$ and $\bm \varphi_q^0$ are defined in Lemma \ref{limz}.
\end{lemma}

\begin{proof}
First, we show the existence. Here, we will only show the existence of $(\nu_1^d,\theta_1^d)$, and the others could be obtained similarly.
Let $$Y_1:=\{\bm x=(x_1,\cdots,x_n)^T\in\mathbb {C}^n:x_1=0\},$$
 and consequently $\mathbb C^n={\rm span}\{\bm\varphi_1^0\}\oplus Y_1$. Let
$$\bm H_1(d, \nu,\theta,\bm \xi_1):=\bm H(d, \nu,\theta,\bm\varphi^0_1+\bm \xi_1):\mathbb R^3\times Y_1\to \mathbb C^n.$$
Clearly, we have $\bm H_1(0,\nu_1^0,\theta_1^0,\bm 0)=\bm 0$, and the Fr\'echet derivative of $\bm H_1$ with respect to
$(\nu,\theta,\bm \xi_1)$ at $(0,\nu_1^0,\theta_1^0,\bm 0)$ is
\begin{equation*}
\begin{split}
D_{(\nu,\theta,\bm \xi_1)}\bm H_1(0,\nu_1^0,\theta_1^0,\bm 0)[\vartheta,\epsilon,\bm \chi]=\left(\begin{array}{c}
\ds -{\rm i} e^{-{\rm i} \theta_1^0}b_{1}^{0}u_1^0\epsilon-{\rm i}\vartheta
\\
\ds \left(a_{2}^{0}u_2^0+b_{2}^{0}u_2^0e^{-{\rm i} \theta_1^0} -{\rm i} \nu_1^0\right) \chi_2\\
\vdots\\
\ds \left(a_{n}^{0}u_n^0+ b_{n}^{0}u_n^0e^{-{\rm i} \theta_1^0} -{\rm i} \nu_1^0\right) \chi_n\\
\end{array}\right),
\end{split}
\end{equation*}
where $\vartheta,\epsilon\in\mathbb R$ and $\bm \chi=(\chi_1,\cdots,\chi_n)\in Y_1$. Note from \eqref{isolp} that $D_{(\nu,\theta,\bm \xi)}\bm H_1(0,\nu_1^0,\theta_1^0,\bm 0)$ is a bijection.
Then from the implicit function theorem, there exists a constant $\delta>0$, a neighborhood $N_1$ of $(\nu_1^0,\theta_1^0,\bm 0)$ and a continuously differentiable function
$$
(\nu_1^d,\theta_1^d,\bm \xi_1^d):[0,\delta)\mapsto N_1
$$
such that for any $d\in [0,\delta)$, the unique solution of $\bm H_1(d, \nu,\theta,\bm \xi_1)=\bm 0$ in the neighborhood $N_1$ is $(\nu_1^d,\theta_1^d,\bm \xi_1^d)$. Letting $\bm \varphi_1^d=\bm \varphi_1^0+\bm \xi_1^d$, we see that
 \begin{equation}\label{SSp}
\textrm{span}(\bm \varphi_1^d)\subset \mathcal S^d_1 \;\;\text{for any}\;\; d\in[0,\delta).
\end{equation}
Since the dimension of $\mathcal S^d_1$ is upper semicontinuous, then there exists $\delta_1<\delta$
such that $\dim \mathcal S_1^d\le 1$ for any $d\in[0,\delta_1)$. This, together with \eqref{SSp}, implies that
$\mathcal S^d_1=\{c\bm\varphi_1^d:c\in\mathbb C\}$. By \eqref{tauthetanu}, we see that $(\nu_q^0,\theta_q^0)\in(0,\infty)\times (0,2\pi)$, which yields $(\nu_q^d,\theta_q^d)\in(0,\infty)\times (0,2\pi)$ for
$0<d\ll1$.
This completes the part of existence.

Now we show that \eqref{4.8} holds.
If it is not true, then there exist sequences $\{d_j\}_{j=1}^\infty$
and $\left\{\left(\nu^{d_j},\theta^{d_j},\bm \varphi^{d_j}\right)\right\}_{j=1}^\infty$
such that
$\lim_{j\to\infty}d_j=0$, and for each $j=1,2,\cdots,$ $\left(\nu^{d_j},\theta^{d_j}\right)\ne (\nu_q^{d_j},\theta_q^{d_j})(q=1,\cdots,p)$, $\left\|\bm \varphi^{d_j}\right\|_2=1$, $\nu^{d_j}>0$, $\theta^{d_j}\in[0,2\pi)$, and
$$\bm H\left(d_j, \nu^{d_j},\theta^{d_j},\bm \varphi^{d_j}\right)=\bm 0.$$
By Lemma \ref{3nu}, we see that $\{\nu^{d_j}\}$ is bounded.
Using similar arguments as in the proof of \cite[Lemma 3.4]{ChenHopf}, we show that
there exists $1\le q_0\le p$ such that  $\left(\nu^{d_j},\theta^{d_j}\right)=(\nu^{d_j}_{q_0},\theta^{d_j}_{q_0})$ for sufficiently large $j$. This is a contradiction. Therefore, \eqref{4.8} holds.
\end{proof}

From Lemma \ref{l4.4}, we obtain the following result.
\begin{theorem}\label{3c25}
Assume that $\bf(H0)$-$\bf(H2)$ and \eqref{isolp} hold, and $d\in(0,\tilde d)$, where $0<\tilde d\ll 1$. Then $(\nu,\tau,\bm \varphi)$ solves
\begin{equation*}
\begin{cases}
\Delta(d,{\rm i}\nu,\tau)\bm \varphi=\bm0,\\
\nu>0,\;\tau\ge0,\;\bm \varphi (\ne \bm0) \in \mathbb C^n,\\
\end{cases}
\end{equation*}
if and
only if there exists $1\le q\le p$ such that
\begin{equation}\label{3par}
\nu=\nu^d_q,\;\bm \varphi= c{\bm \varphi}^d_q,\;
\tau=\tau^d_{q,l}=\frac{\theta^d_q+2l\pi}{\nu^d_q},\;\; l=0,1,2,\cdots,
\end{equation}
where $\nu^d_q$, $\theta^d_q$, and ${\bm \varphi}^d_q$ are defined in Lemma \ref{l4.4}.
\end{theorem}

Then we show that the purely imaginary eigenvalue is simple.

\begin{theorem}\label{smallsimpleeigenvalue}
Assume that $\bf(H0)$-$\bf(H2)$ and \eqref{isolp} hold. Then, for each $d\in(0,\tilde d)$, where $0<\tilde d\ll 1$, $\mu={\rm i} \nu_{q}^{d}$ is a simple eigenvalue of $A_{\tau^d_{q,l}}(d)$ for $q=1, \cdots, p$ and $l=0, 1, 2, \cdots.$
\end{theorem}

\begin{proof}
It follows from Theorem \ref{3c25} that $$\mathscr{N}\left[A_{\tau^d_{q,l}}(d)-{\rm i} \nu_{q}^{d}\right]=\operatorname{span}\left[e^{{\rm i} \nu_{q}^{d} \theta}\bm \varphi_{q}^{d} \right],$$ where $\theta\in\left[-\tau^d_{q,l},0\right]$, and $\bm \varphi_{q}^{d}$ is defined in Theorem \ref{3c25}. Then, we will show that
$$\mathscr{N}\left[A_{\tau^d_{q,l}}(d)-{\rm i} \nu_{q}^{d}\right]^{2}=\mathscr{N}\left[A_{\tau^d_{q,l}}(d)-{\rm i} \nu_{q}^{d}\right].$$
If $\bm \phi \in \mathscr{N}\left[A_{\tau^d_{q,l}}(d)-{\rm i} \nu_{q}^{d}\right]^{2}$, then
$$
\left[A_{\tau^d_{q,l}}(d)-{\rm i} \nu_{q}^{d}\right]\bm \phi \in \mathscr{N}\left[A_{\tau^d_{q,l}}(d)-{\rm i} \nu_{q}^{d}\right]=\operatorname{span}\left[e^{{\rm i} \nu_{q}^{d} \theta}\bm \varphi_{q}^{d} \right],
$$
and consequently, there exists a constant $\gamma$ such that
$$
\left[A_{\tau^d_{q,l}}(d)-{\rm i} \nu_{q}^{d}\right]\bm \phi=\gamma e^{{\rm i} \nu_{q}^{d} \theta}\bm \varphi_{q}^{d},
$$
which yields
\begin{equation}\label{smalldotphi}
\begin{aligned}
\dot{\bm\phi}(\theta) =&{\rm i} \nu_{q}^{d}\bm \phi(\theta)+\gamma e^{{\rm i} \nu_{q}^{d} \theta}\bm \varphi_{q}^{d}, \quad \theta \in\left[-\tau^d_{q,l}, 0\right], \\
\dot{\bm\phi}(0)=&dA\bm \phi(0)+ \operatorname{diag} \left(f_j (u^d_{j},u^d_{j})\right)\bm \phi(0) + \operatorname{diag}\left(u^d_{j} a_{j}^{d}\right)\bm\phi(0)\\
&+ \operatorname{diag}\left(u^d_{j} b_{j}^{d}\right)\bm \phi(-\tau^d_{q,l}).
\end{aligned}
\end{equation}
From the first equation of Eq. \eqref{smalldotphi}, we have
\begin{equation}\label{smallphi}
\begin{aligned} \bm\phi(\theta) &=\bm\phi(0) e^{{\rm i} \nu_{q}^{d} \theta}+\gamma \theta e^{{\rm i} \nu_{q}^{d} \theta}\bm \varphi_{q}^{d}, \\
\dot{\bm\phi}(0) &={\rm i} \nu_{q}^{d}\bm \phi(0)+\gamma \bm \varphi_{q}^{d}.
\end{aligned}
\end{equation}
Then it follows from Eqs. \eqref{smalldotphi} and \eqref{smallphi} that
\begin{equation}\label{smallDeltaphi1}
\begin{aligned}
\Delta\left(d, {\rm i} \nu_{q}^{d}, \tau^d_{q,l}\right)\bm \phi(0)=&dA\bm \phi(0)+ \operatorname{diag} \left(f_j (u^d_{j},u^d_{j})\right)\bm \phi(0) + \operatorname{diag}\left(u^d_{j} a_{j}^{d}\right)\bm\phi(0)\\
&+ e^{{-\rm i} \theta_{q}^{d}}\operatorname{diag}\left(u^d_{j} b_{j}^{d}\right)\bm\phi(0) - {\rm i}\nu_{q}^{d} \bm \phi(0)\\
=& \gamma \left(\bm\varphi_{q}^{d}+ \tau^d_{q,l} e^{{-\rm i} \theta_{q}^{d}}\operatorname{diag}(u^d_{j} b_{j}^{d})\bm \varphi_{q}^{d}\right).
\end{aligned}
\end{equation}

Let $\widetilde{\Delta}\left(d, {\rm i} {\nu_{q}^d}, {\tau}^{d}_{q,l}\right)$ be the conjugate transpose matrix of $\Delta\left(d, {\rm i} \nu_{q}^{d}, \tau^d_{q,l}\right)$, and
let $\widetilde{\bm\varphi}^d_{q}=(\widetilde\varphi_{q,1}^{d},\cdots, \widetilde\varphi_{q,n}^{d})^T$ be the
 the corresponding eigenvector of $\widetilde{\Delta}\left(d, {\rm i} {\nu_{q}^d}, {\tau}^{d}_{q,l}\right)$ with respect to eigenvalue $0$.
Then, using similar arguments as in the proof of Proposition \ref{p}, we see that, ignoring a scalar factor,  $\widetilde{\bm\varphi}^d_{q}$ satisfies
\begin{equation}\label{limwide}
\lim_{d\to0}\widetilde{\bm\varphi}^d_{q}=\bm\varphi^0_{q},
\end{equation}
where $\bm \varphi^0_{q}$ is defined in Lemma \ref{limz}.
Multiplying both sides of \eqref{smallDeltaphi1} by $(\overline{\widetilde\varphi}_{q,1}^{d},\cdots, \overline{\widetilde\varphi}_{q,n}^{d})$ to the left, we have
$$
\begin{aligned}
0 &=\left\langle  \widetilde \Delta\left(d, {\rm i} \nu_{q}^{d},\tau^d_{q,l}\right)\widetilde{\bm \varphi}_{q}^{d}, \bm \phi(0)\right\rangle=\left\langle\widetilde{\bm \varphi}_{q}^{d}, \Delta\left({d}, {\rm i} \nu_{q}^{d}, \tau^d_{q,l}\right)\bm \phi(0)\right\rangle\\
& = \gamma\left(\sum_{j=1}^n\overline{\widetilde\varphi}^d_{q,j}\varphi_{q,j}^{d}+ \tau^d_{q,l} e^{{-\rm i} \theta_{q}^{d}} \sum_{j=1}^n u^d_{j}b_{j}^{d}\overline{\widetilde\varphi}_{q,j}^{d}\varphi^d_{q,j}\right):=\gamma S_{q}(d).
\end{aligned}
$$
It follows from Lemma \ref{l4.4}, Theorem \ref{3c25} and Eq. \eqref{limwide} that
\begin{equation*}
\displaystyle\lim_{d\to0}S_{q}(d)\ne 0.
\end{equation*}
which implies that $\gamma=0$ for $d\in(0,\tilde d)$, where $0<\tilde d\ll 1$,
and consequently, ${\rm i}\nu_{q}^{d}$ is a simple eigenvalue of
$A_{\tau^d_{q,l}}(d)$ for $q=1, \cdots, p$ and $l=0, 1, 2, \cdots$.
\end{proof}

By Theorem \ref{smallsimpleeigenvalue} and the implicit function theorem, we see that, for each $q=1,\cdots,p$ and $l=0,1,2,\cdots$, there exists a neighborhood $O_{q,l}\times D_{q,l}\times H_{q,l}$ of $({\tau^d_{q,l}},\textrm{i}\nu_q^{d},{\bm \varphi}_q^{d})$ and a continuously
differentiable function $(\mu(\tau),\bm\varphi(\tau)):O_{q,l}\rightarrow D_{q,l}\times H_{q,l}$ such that $
\mu(\tau^d_{q,l})=\textrm{i}\nu_q^d$, $\bm \varphi(\tau^d_{q,l})={\bm \varphi}_q^d$, and for each $\tau\in O_{q,l}$, the only eigenvalue of $A_{\tau}(d)$ in $D_{q,l}$ is $\mu(\tau),$ and
\begin{equation}\label{smallDeltapsi}
\begin{split}
\Delta(d,\mu(\tau),\tau)\bm\varphi(\tau)=&dA\bm \varphi(\tau)+ \operatorname{diag} \left(f_j (u^d_{j},u^d_{j})\right)\bm \varphi(\tau) + \operatorname{diag}\left(u^d_{j} a_{j}^{d}\right)\bm\varphi(\tau)\\
&+ e^{-\mu(\tau) \tau}\operatorname{diag}\left(u^d_{j} b_{j}^{d}\right)\bm\varphi(\tau) - \mu(\tau) \bm \varphi(\tau)=\bm 0.
\end{split}
\end{equation}
Then, using similar arguments as Theorem \ref{Remu}, we obtain the following transversality condition.

\begin{theorem}\label{smallRemu}
Assume that $\bf(H0)$-$\bf(H2)$ and \eqref{isolp} hold. Then
$$
\frac{d \mathcal{R} e\left[\mu\left(\tau^d_{q,l}\right)\right]}{d \tau}>0, \;\; q=1,\cdots, p,\;\; l=0,1,2,\cdots.
$$
\end{theorem}

By Theorems \ref{sstability1} and \ref{3c25}-\ref{smallRemu}, we obtain the following result.
\begin{theorem}\label{dsmTh}
Assume that $\bf(H0)$-$\bf(H1)$ hold, and $d\in(0,\tilde d)$, where $0<\tilde d\ll 1$. Let $\bm u^d$ be the unique positive equilibrium obtained in Lemma \ref{GAS}. Then the following statements hold.
\begin{enumerate}
\item [{\rm (i)}] If $a_j^0-b_j^0<0$ for all $j=1,\cdots,n$, then $\bm u^d$ is locally asymptotically stable for $\tau\in[0,\infty)$
\item [{\rm (ii)}] If $\bf(H2)$ and \eqref{strong} holds, then $\bm u^d$ is locally asymptotically stable for $\tau\in[0, \tau^d_{\hat q,0})$, and
unstable for $\tau \in( \tau^d_{\hat q,0},\infty)$, where $\tau^d_{\hat q,0}=\ds\min_{1\le q\le p}\tau_{q,0}^d$. Moreover, when $\tau=\tau^d_{\hat q,0}$, system \eqref{M1} undergoes a Hopf bifurcation.
\end{enumerate}
\end{theorem}



\section{An example}
In this section, we apply the obtained results in Section 3 to a concrete example and discuss the effect of network topology on Hopf bifurcations.
Choose the growth rate per capita as follows: $$f_j(u_j, u_j(t-\tau))=m_j- \hat a_j u_j(t)- \hat b_j u_j(t-\tau) \;\;\;\text{for}\;\;\; j=1, \cdots, n.$$
 Then model \eqref{M1} takes the following form:
\begin{equation}\label{M12}
\begin{cases}
\displaystyle \frac{d u_{j}}{d t}=d \sum_{k=1}^{n} \alpha_{jk} u_{k}+ u_j \left(m_j- \hat a_j u_j(t)- \hat b_j u_j(t-\tau)\right),  &t>0,~~ j=1, \cdots, n,\\
\displaystyle \bm{u}(t)=\bm\psi(t) \geq \bm 0, & t \in[-\tau, 0],
\end{cases}
\end{equation}
where $(\alpha_{jk})$ satisfies assumption $\bf(H0)$, $m_j$ represents the intrinsic growth rate in patch $j$,
and $\hat a_j,\hat b_j>0$ represent the instantaneous and delayed dependence of the growth rate in patch $j$, respectively. Clearly, assumption $\bf(H1)$ holds. We remark that the continuous space version of model \eqref{M12} with spatially homogeneous environments  has been investigated in \cite{SuWeiShi2012}.
\subsection{Stability and Hopf bifurcations}
For case (I) ($0<d_*-d\ll1$), the quantities $\tilde a$ and $\tilde b$  take  the following form:
\begin{equation}\label{tilder1r21}
\tilde a=-\sum_{j=1}^{n}\hat a_j \eta_j^2 \varsigma_j,\;\;\tilde b=-\sum_{j=1}^{n}\hat b_j \eta_j^2 \varsigma_j,
\end{equation}
where $\bm\eta$ and $\bm\va$ are defined in \eqref{bmeta}. Then, by Theorem
\ref{stable}, we obtain the following result.
\begin{proposition}
Let $\bm u^d$ be the unique positive equilibrium of \eqref{M1} obtained in Lemma \ref{GAS} for $d\in(0,d_*)$. Then, for $d \in [\tilde d_2, d_*)$ with $0<d_*-\tilde d_2 \ll 1$, the following statements hold.
 \begin{enumerate}
    \item [{\rm (i)}] If $\sum_{j=1}^n \left(\hat a_j-\hat b_j\right)\eta_j^2 \va_j>0 $, then $\bm {u}^{d}$ of model \eqref{M1} is locally asymptotically stable for $\tau \in\left[0, \infty\right)$.
    \item [{\rm (ii)}] If $\sum_{j=1}^n \left(\hat a_j-\hat b_j\right)\eta_j^2 \va_j<0 $, then $\bm {u}^{d}$  is locally asymptotically stable for $\tau \in\left[0, \tau_{d,0}\right),$ and unstable for $\tau \in\left(\tau_{d,0}, \infty\right)$, where $\tau_{d,0}$ is defined in Theorem \ref{solutionvtaupsi}. Moreover, when $\tau=\tau_{d,0},$ system \eqref{M12} undergoes a Hopf bifurcation at $\bm {u}^{d}$.
    \end{enumerate}
\end{proposition}

Now we consider case (II) ($(0<d\ll1$). The quantities for this case take the following form:
\begin{equation}\label{tilder1r211}
 a_j^0=-\hat a_j,\;\; b_j^0=-\hat b_j, \;\;\nu_j^0=\frac{m_j \sqrt{(\hat b_j)^2-(\hat a_j)^2}}{\hat a_j+\hat b_j}, \;\; \theta_j^0= \arccos \left(-\frac{\hat a_j  }{ \hat b_j }\right).
\end{equation}
Moreover, $\bf(H2)$ is reduced as follows:
\begin{enumerate}
\item [$\bf(\tilde{H}2)$] $\hat a_j-\hat b_j< 0$ for $j=1,\cdots,p$, and $\hat a_j-\hat b_j>0$ for $j=p+1,\cdots,n$, where $1\le p\le n$.
\end{enumerate}
Then, by Theorem \ref{dsmTh}, we have the following result.

\begin{proposition}\label{3c25new1}
Let $\bm u^d$ be the unique positive equilibrium of \eqref{M1} obtained in Lemma \ref{GAS} for $d\in(0,d_*)$. Then, for $d \in (0, \tilde d)$ with $0<\tilde d \ll 1$, the following statements hold.
\begin{enumerate}
\item [{\rm (i)}] If $\hat a_j-\hat b_j>0$ for all $j=1,\cdots,n$, then $\bm u^d$ is locally asymptotically stable for $\tau\in[0,\infty)$.
\item [{\rm (ii)}] If $\bf(\tilde{H}2)$ holds and $\ds\frac{\theta_j^0}{\nu_j^0}\ne\frac{\theta_k^0}{\nu_k^0}$ for any $j\ne k$ and $1\le j,k\le p$, then $\bm u^d$ is locally asymptotically stable for $\tau\in[0, \tau^d_{\hat q,0})$, and
unstable for $\tau \in( \tau^d_{\hat q,0},\infty)$, where $\tau^d_{\hat q,0}$ is defined in Theorem \ref{dsmTh}. Moreover, when $\tau=\tau^d_{\hat q,0}$, system \eqref{M1} undergoes a Hopf bifurcation at $\bm {u}^{d}$.
\end{enumerate}
\end{proposition}
\begin{remark}
We remark that Proposition \ref{3c25new1} (ii) also holds if $\bf(\tilde{H}2)$ is replaced by the following assumption:
\begin{enumerate}
\item [$\bf(\tilde{A}2)$] $\{\hat a_j-\hat b_j\}_{j=1}^n$ changes sign and $\hat a_j-\hat b_j\ne0$ for all $j=1,\cdots,n$.
\end{enumerate}
The proof is similar, and here we omit the details for simplicity.
\end{remark}

\subsection{The effect of network topologies}
In this subsection, we discuss the effect of network topologies on Hopf bifurcations values for $0<d\ll1$. Since the computation is tedious, we only consider a special case for simplicity. Letting $\hat a_j=0$ and $\hat b_j=1$ for $j=1,\cdots,n$, model \eqref{M12} is reduced to the following system:
\begin{equation}\label{simple}
\begin{cases}
\displaystyle \frac{d u_{j}}{d t}=d \sum_{k=1}^{n} \alpha_{jk} u_{k}+ u_j \left(m_j- u_j(t-\tau)\right),  &t>0,~~ j=1, \cdots, n,\\
\displaystyle \bm{u}(t)=\bm\psi(t) \geq \bm 0, & t \in[-\tau, 0],
\end{cases}
\end{equation}
where $(\alpha_{jk})$ satisfies assumption $\bf(H0)$, and $m_j>0$ for $j=1, \cdots, n$. Clearly, $\bf(H1)$-$\bf(H2)$ hold. By Proposition \ref{3c25new1} (ii) and a direct computation, we see that,  if
\begin{equation}\label{mjmk}
m_j\ne m_k\;\;\text{for any}\;\; j\ne k,
\end{equation}
then model \eqref{simple} undergoes a Hopf bifurcation for $0<d\ll 1$ with the first Hopf bifurcation value $\tau=\tau^d_{\hat q,0}$, where $\hat q$ satisfies $m_{\hat q}=\ds\max_{1\le j\le n} m_j$. By Lemma \ref{l4.4} and Theorem \ref{3c25}, we see that
\begin{equation}\label{firHo}
\tau^d_{\hat q,0}=\ds\frac{\theta_{\hat q}^d}{\nu_{\hat q}^d}\;\;\text{and}\;\; \lim_{d \to 0}\tau^d_{\hat q,0}=\ds\f{\pi}{2  m_{\hat q}}.
\end{equation}
Therefore, to obtain the effect of network topologies, we need to compute
the first derivative of $\tau^d_{\hat q,0}$ with respect to $d$ in the following.
\begin{proposition}\label{Fina}
Let $\tau^d_{\hat q,0}$ be defined in \eqref{firHo}, where $\hat q$ satisfies $m_{\hat q}=\max_{1\le j\le n} m_j$. Then
\begin{equation}\label{2firHo1}
\left(\tau^d_{\hat q,0} \right)'\big|_{d=0}=\ds\f{\mathcal T(A)}{ m^2_{\hat q}},
\end{equation}
where
\begin{equation}\label{matht}
\mathcal T(A):=-\ds\f{\pi}{2}\alpha_{{\hat q}{\hat q}}+\left(1-\ds\f{\pi }{2}\right)\ds\f{1}{m_{\hat q}}\sum_{k\ne{\hat q}} {\alpha_{{\hat q}k}m_k}.
\end{equation}
\end{proposition}

\begin{proof}
By \eqref{firHo}, we have
\begin{equation}\label{rk0}
\left({\tau^d_{\hat q,0}}\right)'=\displaystyle\left(\frac{\theta_{\hat q}^d}{\nu_{\hat q}^d}\right)'=\frac{\left({\theta_{\hat q}^d}\right)'{\nu_{\hat q}^d}-{\theta_{\hat q}^d}\left({\nu_{\hat q}^d}\right)'}{\left({\nu_{\hat q}^d}\right)^2},
\end{equation}
where $'$ is the derivative with respect to $d$.
Substituting $\nu=\nu_{\hat q}^{d}$, $\theta=\theta_{\hat q}^{d}$ and $\bm \varphi=\bm \varphi_{\hat q}^{d}$ into \eqref{Deltaiv}, we have
\begin{equation}\label{smallDeltaphi111}
\begin{aligned}
dA\bm {\bm\varphi}^d_{\hat q}+ \operatorname{diag} \left(m_j-u^d_{j}\right){\bm\varphi}^d_{\hat q} - e^{{-\rm i} \theta_{\hat q}^{d}}\operatorname{diag}\left(u^d_{j}\right){\bm\varphi}^d_{\hat q} - {\rm i}\nu_{\hat q}^{d} {\bm\varphi}^d_{\hat q}=\bm 0.
\end{aligned}
\end{equation}
Differentiating \eqref{smallDeltaphi111} with respect to $d$, we have
\begin{equation}\label{derivt}
\begin{split}
-\Delta\left(d, {\rm i} \nu_{\hat q}^{d}, \tau^d_{\hat q,0}\right)\left({\bm\varphi}^d_{\hat q}\right)'=&A\bm {\bm\varphi}^d_{\hat q}
-\operatorname{diag} \left((u^d_{j})'\right){\bm\varphi}^d_{\hat q}+{\rm i} (\theta_{\hat q}^{d})'e^{{-\rm i} \theta_{\hat q}^{d}}\operatorname{diag}\left(u^d_{j}\right){\bm\varphi}^d_{\hat q} \\&-e^{{-\rm i} \theta_{\hat q}^{d}}\operatorname{diag}\left((u^d_{j})'\right){\bm\varphi}^d_{\hat q} - {\rm i}(\nu_{\hat q}^{d})' {\bm\varphi}^d_{\hat q},
\end{split}
\end{equation}
where $\Delta(d,\mu,\tau)$ is defined in \eqref{2triangle}.
Let $\widetilde{\bm\varphi}^d_{\hat q}=(\widetilde\varphi_{\hat q,1}^{d},\cdots, \widetilde\varphi_{\hat q,n}^{d})^T$ be the
the corresponding eigenvector of $\widetilde{\Delta}\left(d, {\rm i} {\nu_{\hat q}^d}, {\tau}^{d}_{\hat q,0}\right)$ with respect to eigenvalue $0$, where $\widetilde{\Delta}\left(d, {\rm i} {\nu_{\hat q}^d}, {\tau}^{d}_{\hat q,0}\right)$ is the conjugate transpose matrix of $\Delta\left(d, {\rm i} \nu_{\hat q}^{d}, \tau^d_{\hat q,0}\right)$.
Using similar arguments as in the proof of Proposition \ref{p}, we see that, ignoring a scalar factor,  $\widetilde{\bm\varphi}^d_{\hat q}$ satisfies
\begin{equation}\label{limwide12}
\lim_{d\to0}\widetilde{\bm\varphi}^d_{\hat q}=\bm\varphi^0_{\hat q},
\end{equation}
where $\bm \varphi^0_{q}$ is defined in Lemma \ref{limz}. Note that
$$
0=\left\langle  \widetilde \Delta\left(d, {\rm i} \nu_{\hat q}^{d},\tau^d_{\hat q,0}\right)\widetilde{\bm \varphi}_{\hat q}^{d}, \left({\bm\varphi}^d_{\hat q}\right)'\right\rangle=\left\langle\widetilde{\bm \varphi}_{\hat q}^{d}, \Delta\left({d}, {\rm i} \nu_{\hat q}^{d}, \tau^d_{\hat q,0}\right)\left({\bm\varphi}^d_{\hat q}\right)'\right\rangle.
$$
Then, multiplying both sides of \eqref{derivt} by $(\overline{\widetilde\varphi}_{\hat q,1}^{d},\cdots, \overline{\widetilde\varphi}_{\hat q,n}^{d})$ to the left, we have
\begin{equation}\label{4.7}
\begin{aligned}
0=&-\left\langle \widetilde{\bm\varphi}^d_{\hat q}, \Delta\left(d, {\rm i} \nu_{\hat q}^{d}, \tau^d_{\hat q,0}\right)\left({\bm\varphi}^d_{\hat q}\right)' \right\rangle\\
=&  \left\langle \widetilde{\bm\varphi}^d_{\hat q}, A\bm {\bm\varphi}^d_{\hat q}\right\rangle
-\left\langle \widetilde{\bm\varphi}^d_{\hat q},\operatorname{diag} \left((u^d_{j})'\right){\bm\varphi}^d_{\hat q}\right\rangle+{\rm i} (\theta_{\hat q}^{d})'e^{{-\rm i} \theta_{\hat q}^{d}}\langle\widetilde{\bm\varphi}^d_{\hat q},\operatorname{diag}\left(u^d_{j}\right){\bm\varphi}^d_{\hat q}\rangle\\
 &-e^{{-\rm i} \theta_{\hat q}^{d}}\left\langle\widetilde{\bm\varphi}^d_{\hat q},\operatorname{diag}\left((u^d_{j})'\right){\bm\varphi}^d_{\hat q}\right\rangle - {\rm i}(\nu_{\hat q}^{d})'\left\langle\widetilde{\bm\varphi}^d_{\hat q},{\bm\varphi}^d_{\hat q}\right\rangle.
\end{aligned}
\end{equation}

It follows from Lemma \ref{asymp} that $\bm u^d$ is continuously differentiable for $d\in[0,d_*)$, if we define $u_j^{0}=m_j$ for $j=1,\cdots,n$.
A direct computation yields
\begin{equation}\label{ui0'}
\left(u_{j}^d\right)'\big|_{d=0}=\ds\frac{1}{m_j}\sum_{k=1}^{n} \alpha_{jk} m_k.
\end{equation}
By \eqref{limwide12} and Lemma \ref{l4.4}, we have
\begin{equation}\label{3}
{\bm\varphi}^d_{\hat q},\widetilde{\bm \varphi}^d_{\hat q} \to {\bm\varphi}^0_{\hat q},\;\;\nu_{\hat q}^{d}\to m_{\hat q}, \text{and}\;\; \theta_{\hat q}^{d}\to \theta_{\hat q}^{0}=\frac{\pi}{2} \;\;\text{as}\;\; d\to 0,
\end{equation}
where
${\bm \varphi}^{0}_{\hat q}$ satisfies
$\varphi^{0}_{{\hat q},{\hat q}}=1$ and $\varphi^{0}_{{\hat q},k}=0$ for $k\ne {\hat q}$.
 This, combined with \eqref{4.7} and \eqref{ui0'}, implies that
\begin{equation}\label{1}
(\theta_{\hat q}^{d})'\big|_{d=0}=\frac{1}{m_{\hat q}^2}\sum_{k\ne{\hat q}}\alpha_{\hat q k}m_k\;\;\text{and}\;\;(\nu_{\hat q}^{d})'\big|_{d=0} =\ds\frac{1}{m_{\hat q}}\sum_{k=1}^{n} \alpha_{{\hat q}k} m_k.
\end{equation}
Substituting \eqref{1} into \eqref{rk0}, we obtain that \eqref{2firHo1} holds.
This completes the proof.
\end{proof}
Therefore, for $0<d\ll1$ and a given dispersal matrix $A$, we obtain from \eqref{firHo} and  \eqref{2firHo1} that
\begin{equation}\label{esti}
\tau^d_{\hat q,0} = \ds\f{\pi}{2 m_{\hat q}}+ \ds\f{d}{ m^2_{\hat q}}\mathcal T(A)+\mathcal O(d^2),
\end{equation}
where $\mathcal T(A)$ is defined in \eqref{matht}.

Then, by Proposition \ref{Fina}, we obtain the effect of network topologies as follows.
\begin{proposition}\label{5.3}
Let $\tau^d_{\hat q,0} \left(A_i\right)$ be the first Hopf bifurcation of model \eqref{simple} for $A=A_i$, where $A_i=\left(\alpha_{jk}^{(i)}\right)$ ($i=1,2$) satisfies $\bf(H0)$. If $\mathcal T(A_1)>\mathcal T(A_2)$, then
there is $\hat d>0$, depending on $A_1$ and $A_2$, such that $\tau^d_{\hat q,0}(A_1)>\tau^d_{\hat q,0}(A_2)$ for $d \in\left(0,\hat d \right]$.
\end{proposition}
\begin{remark}
We remark that if $\alpha_{\hat q k}^{(1)}<\alpha_{\hat q k}^{(2)}$ for all $k=1,\cdots,n$, then $\mathcal T(A_1)>\mathcal T(A_2)$.
\end{remark}
By Proposition \ref{Fina}, we can also show the monotonicity of $\tau^d_{\hat q,0}$ for $0<d\ll1$.
\begin{proposition}\label{Fina1}
Let $\tau^d_{\hat q,0}$ be the first Hopf bifurcation of model \eqref{simple},
where $\hat q$ satisfies $m_{\hat q}=\max_{1\le j\le n} m_j$. Then the following statements hold.
\begin{enumerate}
\item [\rm (i)] If $\mathcal T(A)>0$, then $\left({\tau^d_{\hat q,0}}\right)'>0$ for $0<d \ll1$.
\item [\rm (ii)] If $\mathcal T(A)<0$, then $\left({\tau^d_{\hat q,0}}\right)'<0$ for $0<d \ll1$.
\end{enumerate}
Therefore, network topologies also affect the monotonicity of $\tau^d_{\hat q,0}$ for $0<d\ll1$.
\end{proposition}

\subsection{Numerical simulations}
Now, we give some numerical simulations to illustrate our theoretical results for model \eqref{simple}.
Let $n=4$ and $(m_j)=(7.5, 7, 6.5, 6)$, and choose the
following two dispersal matrices:
\begin{equation*}
A_1=\left(\alpha_{jk}^{(1)}\right)=\left( {\begin{array}{*{20}{c}}
{-1}&{0.2}&{0.5}&{0.6}\\
{0.5}&{-1.2}&{0.2}&{0.1}\\
{0}&{0.1}&{-0.9}&{0.1}\\
{0}&{0.1}&{0.2}&{-1.2}
\end{array}} \right),
\end{equation*}
and
\begin{equation*}
A_2=\left(\alpha_{jk}^{(2)}\right)=\left( {\begin{array}{*{20}{c}}
{-2}&{0.2}&{0.5}&{0}\\
{0.5}&{-1.2}&{0.2}&{0.1}\\
{0}&{0.1}&{-0.9}&{0.1}\\
{0}&{0.1}&{0.2}&{-1.2}
\end{array}} \right).
\end{equation*}
Then corresponding network topologies with respect to $A_1$ and $A_2$ are different, see Fig. \ref{patchtu}.

\begin{figure}[t]
\begin{center}
\includegraphics[width=0.4\textwidth,trim=50 0 80 0,clip]{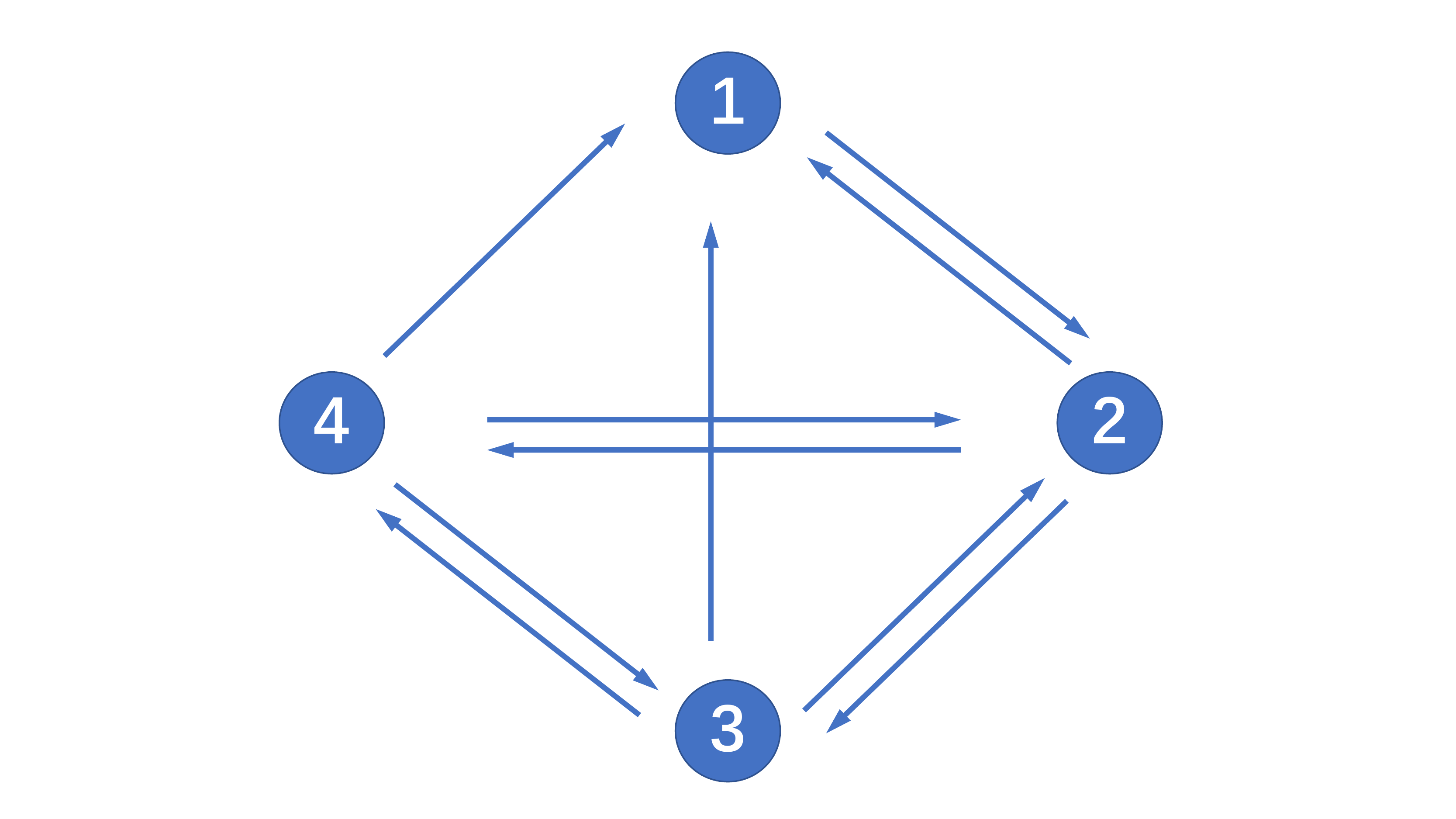}
\includegraphics[width=0.4\textwidth,trim=80 0 50 0,clip]{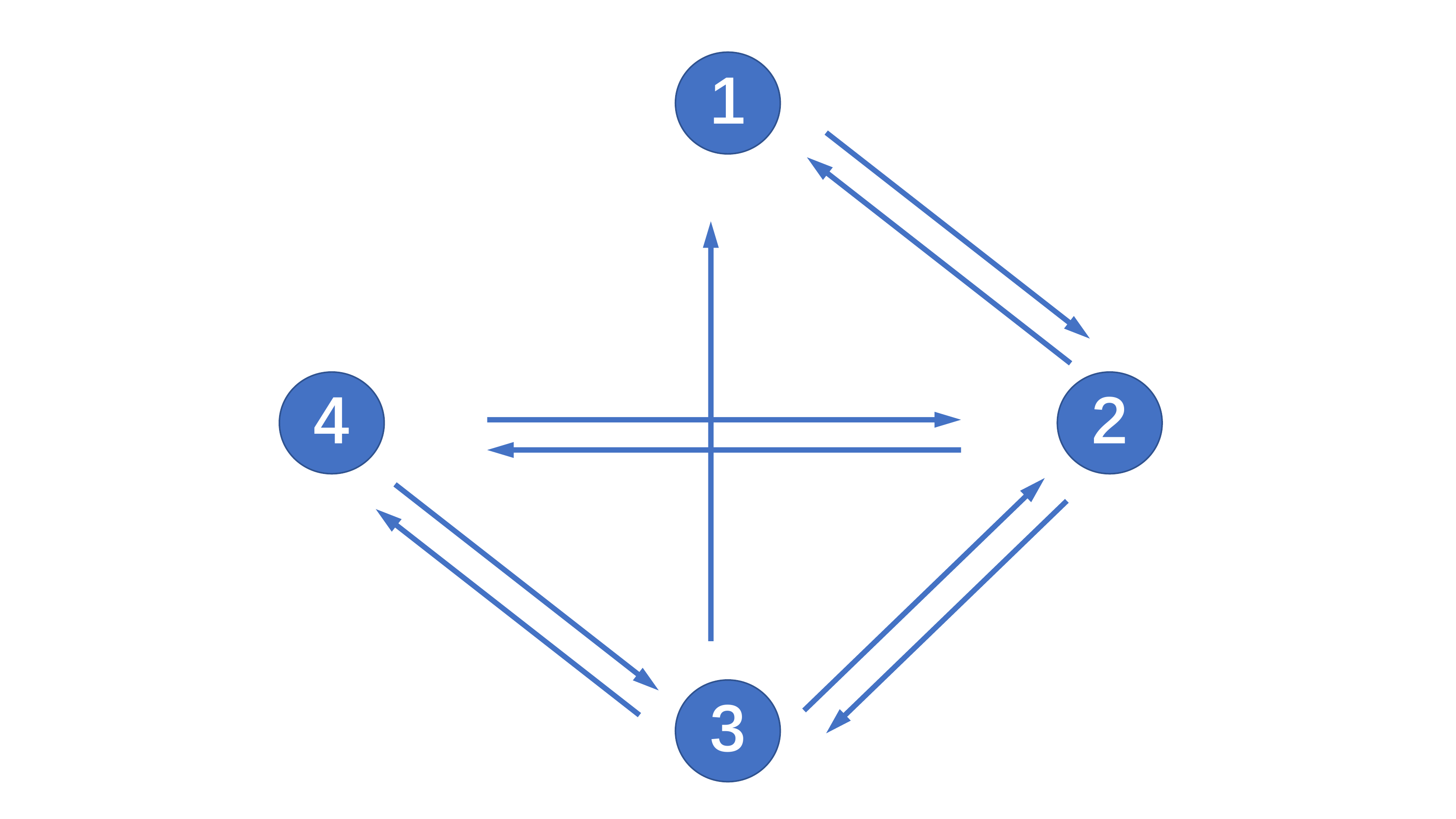}
\caption{Two different network topologies. (Left): $A=A_1$; (Right): $A=A_2$.}
\label{patchtu}
\end{center}
\end{figure}
We first choose $A=A_1$, and numerically show that delay $\tau$ can induce a Hopf bifurcation, and periodic solutions
can occur when  $0<d\ll1$ or $0<d_*-d\le 1$, see Fig. \ref{fig2}.
\begin{figure}[htbp]
\centering\includegraphics[width=0.49\textwidth,trim=100 245 100 210,clip]{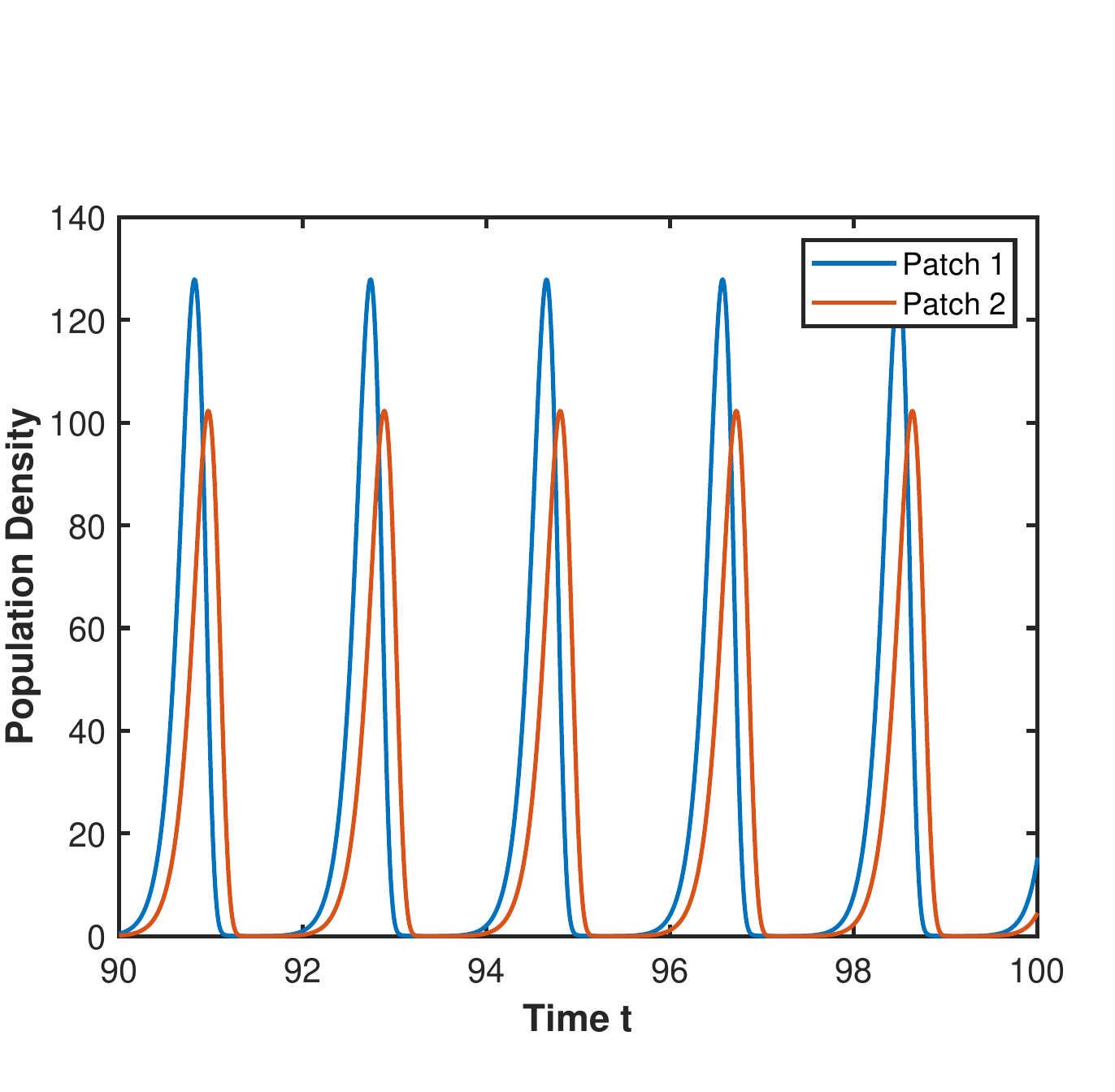}
\includegraphics[width=0.49\textwidth,trim=100 245 100 210,clip]{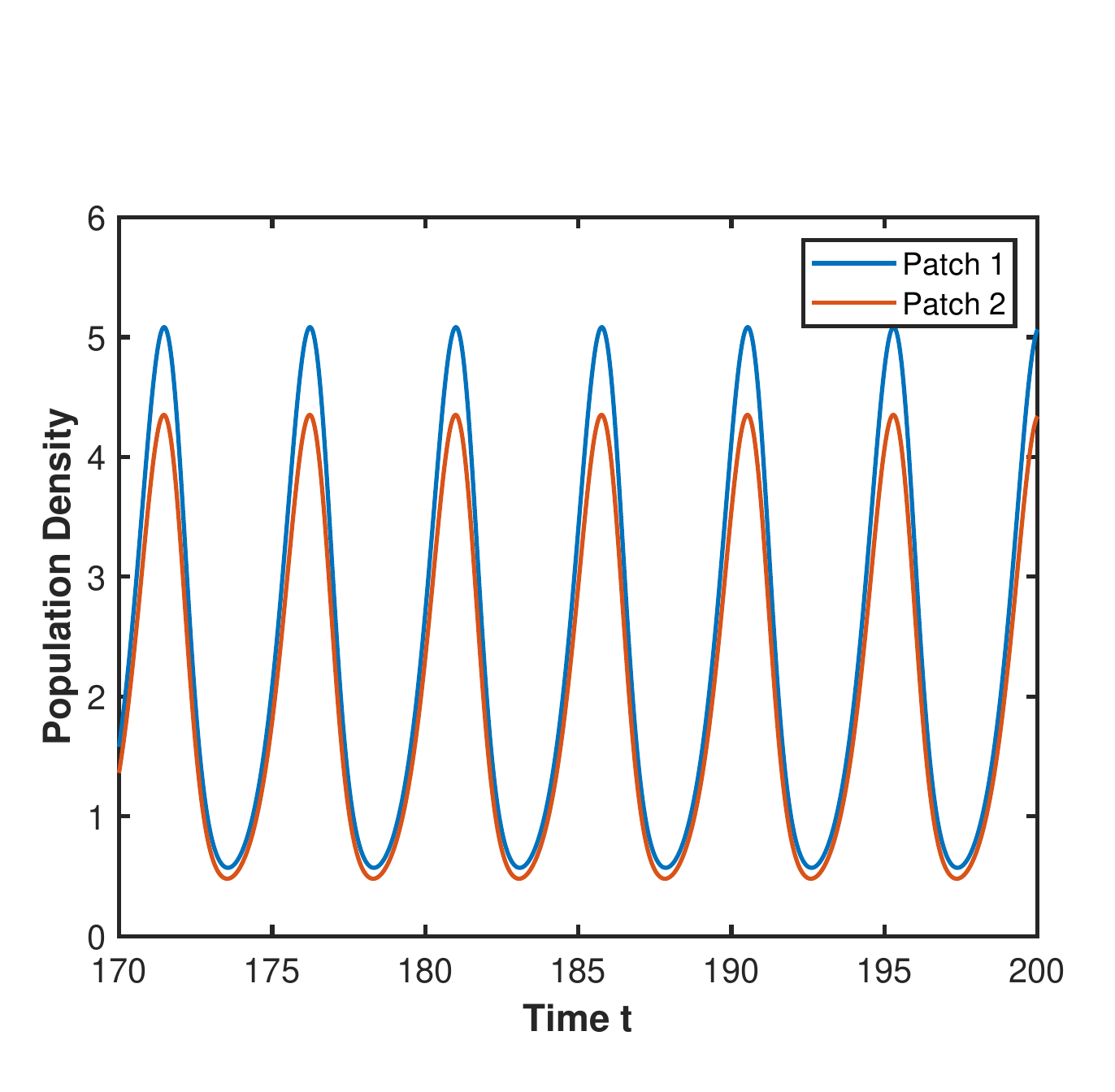}
\caption{Periodic solutions induced by a Hopf bifurcation for model \eqref{simple} with $A=A_1$. (Left) The small dispersal case: $d=0.3$ and $\tau=0.5$. (Right) The large dispersal case: $d = 10$ and $\tau = 1.2$.
  \label{fig2}}
\end{figure}
Then we discuss the effects of network topologies.
Clearly, $\hat q=1$ and $\mathcal T(A_1)>\mathcal T(A_2)$, where $\mathcal T(A)$ is defined in \eqref{matht}. This, combined with
Proposition \ref{Fina}, implies that $\tau^d_{1,0}(A_1)>\tau^d_{1,0}(A_2)$.
To confirm this, we  fix $\tau_1(=0.2144)$, and numerically show that the positive equilibrium of model \eqref{simple} is stable with $A=A_1$, while
model \eqref{simple} admits a positive periodic solution with $A=A_2$,
see Fig. \ref{smalldeltaiii}. Therefore, $\tau^d_{1,0}(A_1)>\tau^d_{1,0}(A_2)$.
\begin{figure}[!htb]
\begin{center}
\includegraphics[width=0.5\textwidth,trim=100 220 90 220,clip]{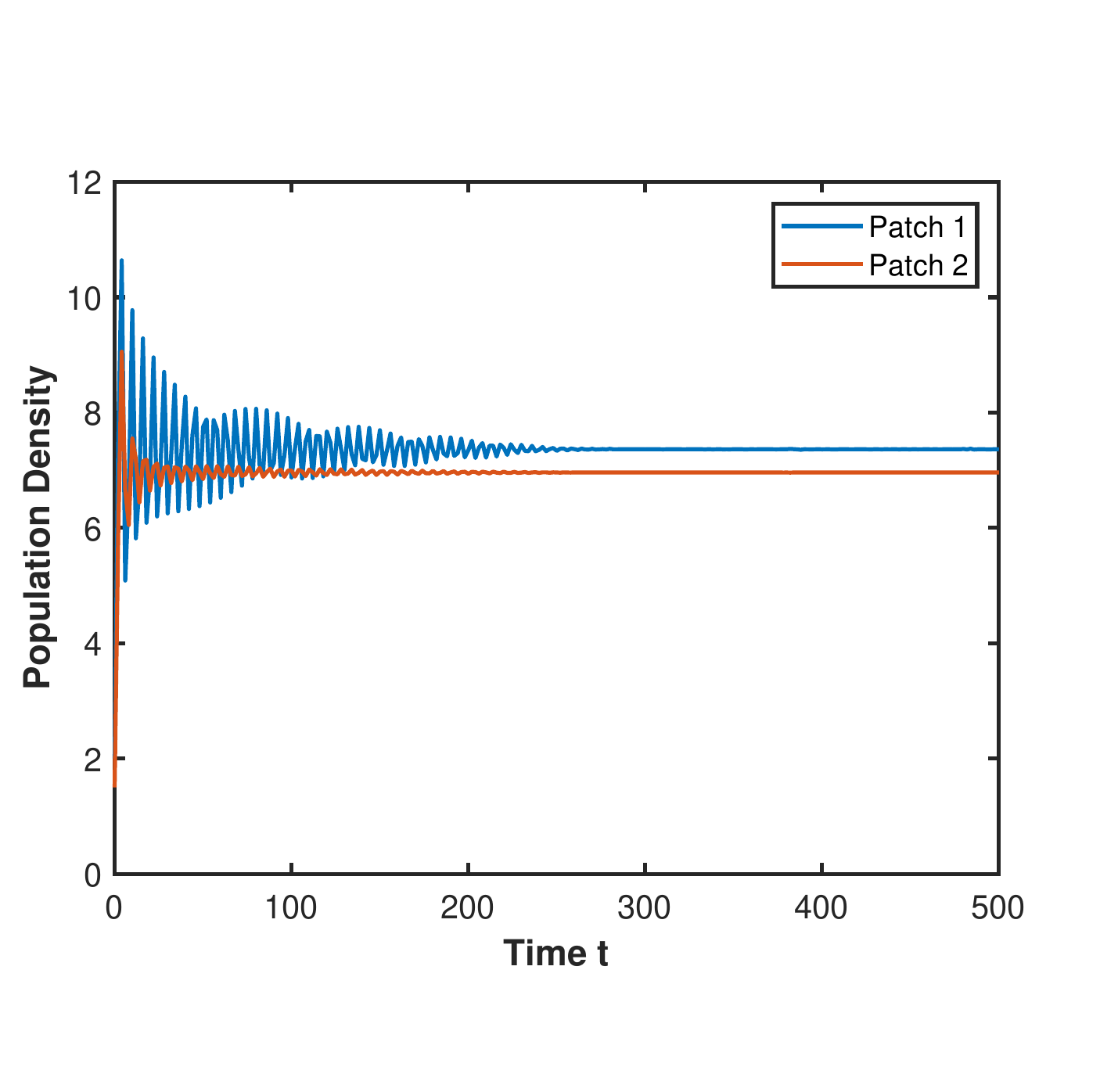}%
\includegraphics[width=0.5\textwidth,trim=90 220 100 220,clip]{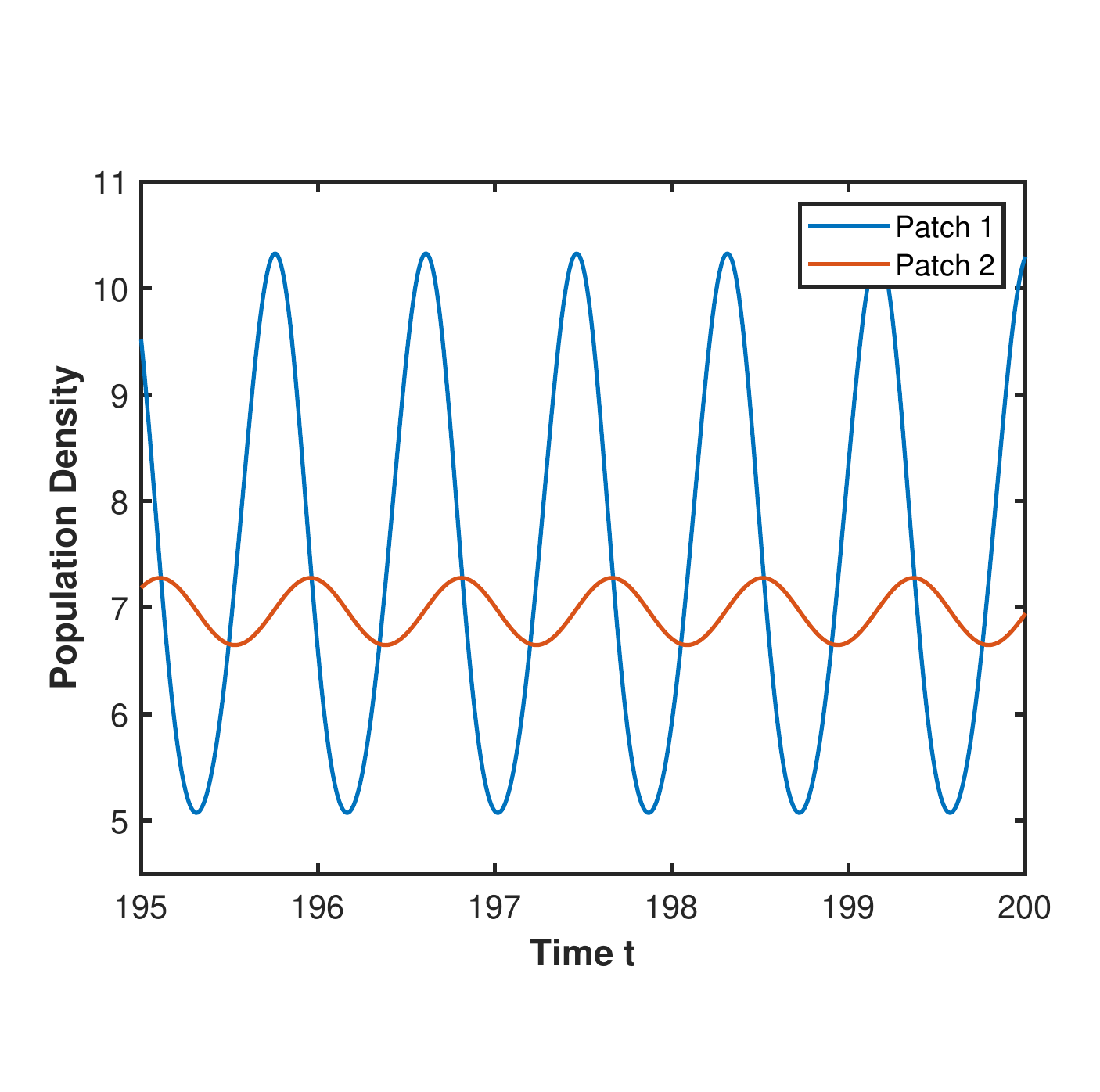}
\caption{The effect of network topologies. We only plot two patches for simplicity. Here $\tau_1=0.2144$. (Left): $A=A_1$; (Right): $A=A_2$.}
\label{smalldeltaiii}
\end{center}
\end{figure}


Moreover, an interesting question is whether Hopf bifurcation can occur when $d$ is intermediate. It is challenge if $n\ge3$.
For the two-patch model, one can compute the Hopf bifurcation value $\tau_{\hat q,0}^d$ for $d\in (0,d_*)$, see \cite{LiaoLou2014} with a symmetric dispersal matrix. Now we consider the asymmetric case. Let
$(m_1,m_2)=(1,2)$, and choose the
following two dispersal matrices:
\begin{equation*}
A_3=\left( {\begin{array}{*{20}{c}}
{-2}&{1}\\
{0.9}&{-1}
\end{array}} \right),\;\;A_4=\left( {\begin{array}{*{20}{c}}
{-20}&{1}\\
{15}&{-1}
\end{array}} \right).
\end{equation*}
For $A=A_i$ with $i=3,4$, we numerically obtain a Hopf bifurcation curve $\tau_{\hat q,0}^d(A_i)$, respectively. Here
 $\lim_{d\to 0}\tau_{\hat q,0}^d(A_i)={\pi}/{4}$ and $\lim_{d\to d_{*}^{(i)}}\tau_{\hat q,0}^d(A_i)=\infty$ with $d_{*}^{(i)}$ satisfies $s(d_{*}^{(i)}A_i+\operatorname{diag}(m_j))=0$ for $i=3,4$.
By Proposition \ref{Fina1}, we see that network topologies also affect the monotonicity of ${\tau^d_{\hat q,0}}$ for $0<d\ll1$.
As is showed in Fig. \ref {fig3}, ${\tau^d_{\hat q,0}}(A_3)$ is monotone increasing for $0<d\ll1$ with $\mathcal T(A_3)=1.3139>0$,  and ${\tau^d_{\hat q,0}}(A_4)$ is monotone decreasing for $0<d\ll1$ with $\mathcal T(A_4)= -2.7102<0$.
\begin{figure}[htbp]
\centering\includegraphics[width=0.496\textwidth,trim=100 245 100 240,clip]{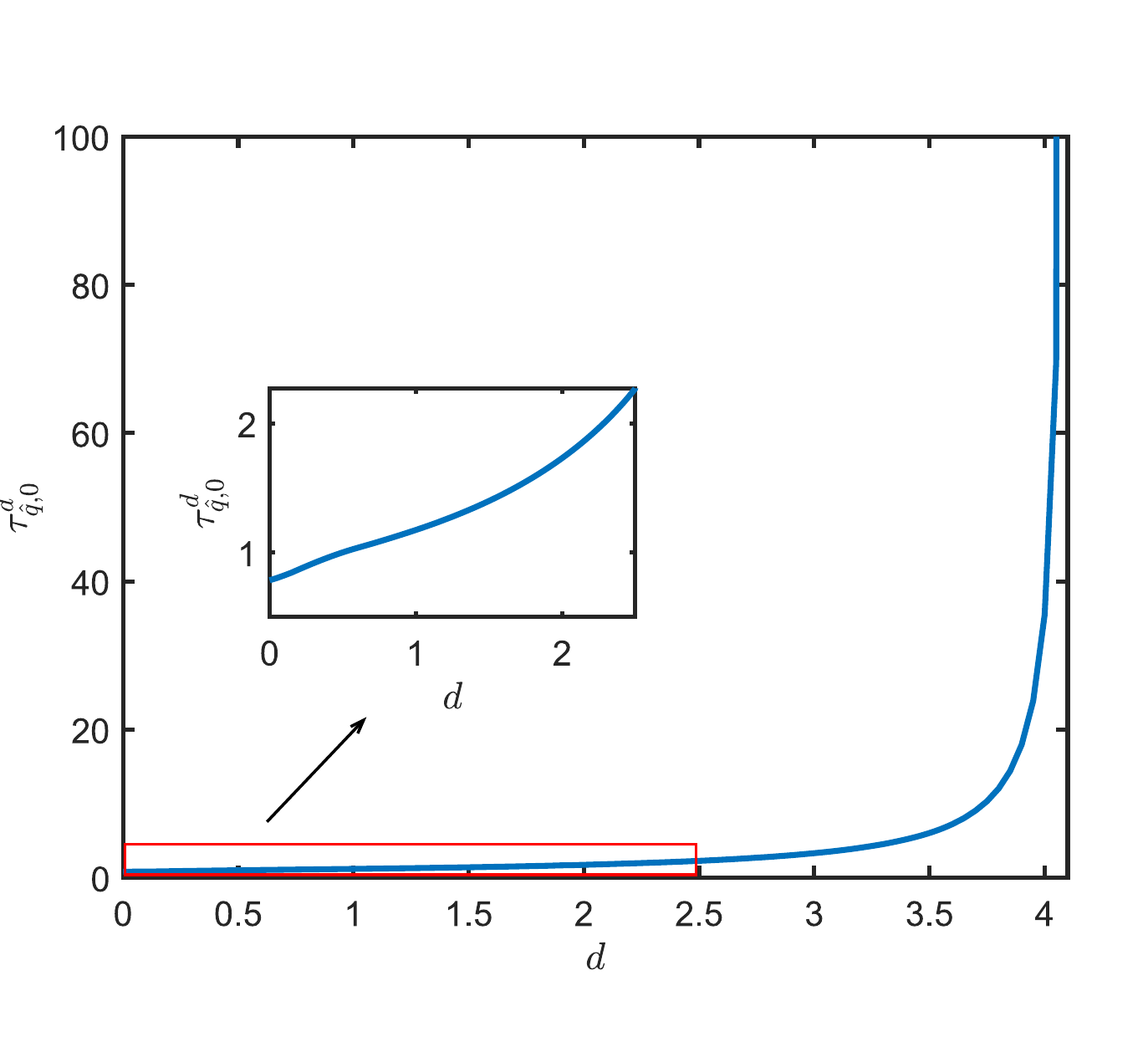}
\centering\includegraphics[width=0.496\textwidth,trim=100 245 100 240,clip]{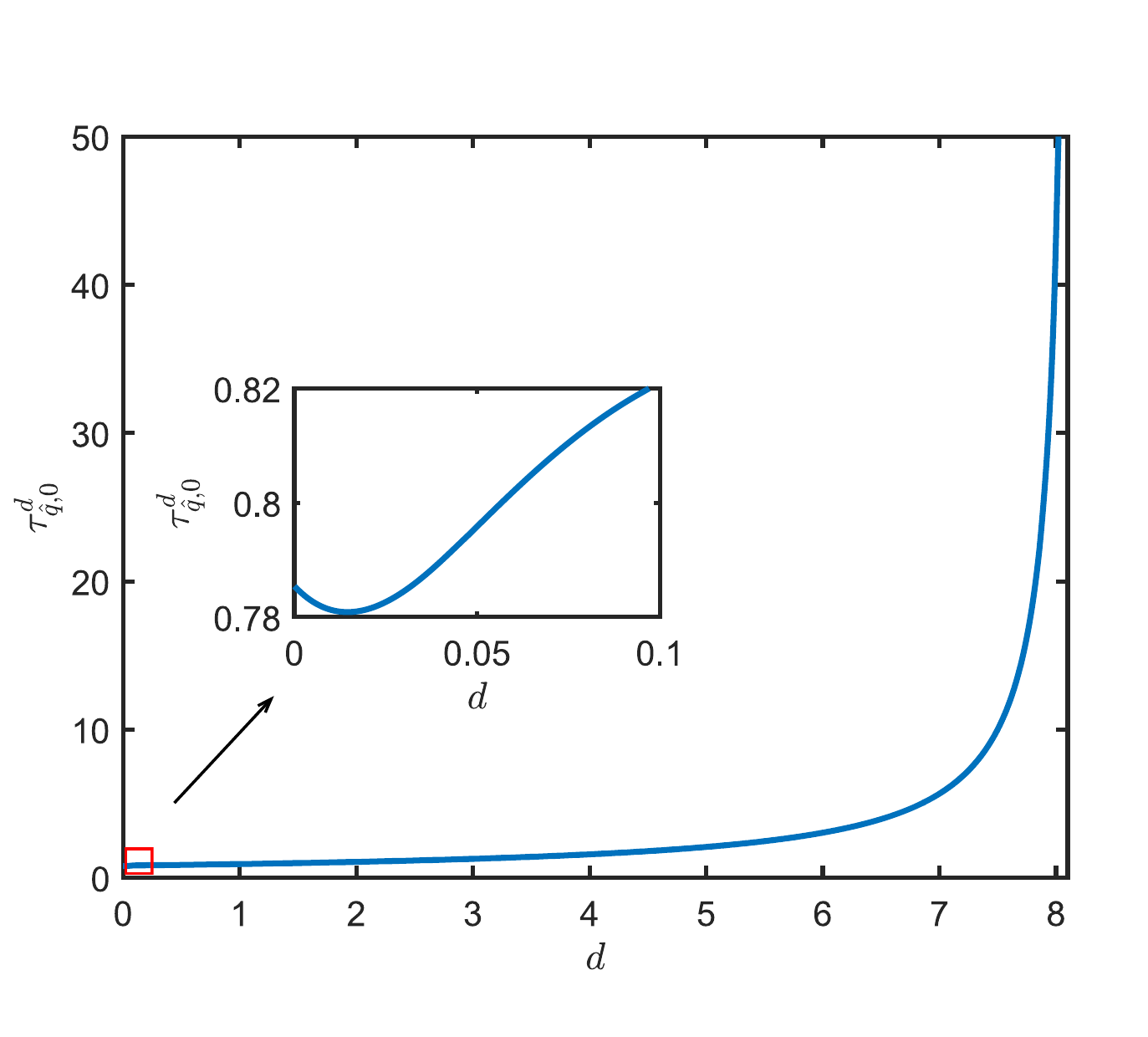}
\caption{The Hopf bifurcation curve. (Left): $A=A_3$; (Right): $A=A_4$.
  \label{fig3}}
\end{figure}
\section{Discussion}
Due to the limits of the method, we only show the existence of a Hopf bifurcation for two cases:
(I) $0<d_*-d\ll 1$, and (II) $0<d\ll 1$.

For case (I), $\tilde a-\tilde b$ is critical to determine the existence of a Hopf bifurcation. We remark that $\tilde a$ and $\tilde b$ are usually negative (see model \eqref{M12} for example), where $-\tilde a>0$ and $-\tilde b>0$ represent the instantaneous and delayed dependence of the growth rate, respectively.
Therefore, $\tilde a-\tilde b<0$ means that the instantaneous term  is dominant, and consequently, delay-induced Hopf bifurcations cannot occur; and $\tilde a-\tilde b>0$
 means that the delay term is dominant, and consequently, delay-induced Hopf bifurcations can occur.
 By \eqref{ula}, we conjecture that $\bm v(t)$ in \eqref{2linear} can be represented as follows:
\begin{equation}\label{near}
\bm v(t)= (d_*-d)\left[c(t)\bm \eta+(d_*-d){\bm z}(t)\right],\;\;\text{where}\;\;c(t)\in\mathbb R\;\; \text{and}\;\; \bm z(t)\in X_1.
\end{equation}
Substituting \eqref{ula} into \eqref{2linear}, we rewrite \eqref{2linear} as follows:
\begin{equation}\label{2linearc}
\begin{split}
\bm v'(t)=&\left[\ds\f{d}{d_*}\left(d_*A+{\rm diag}(m_j)\right)\right]\bm v(t)+(d_*-d){\rm diag}\left(\frac{m_j}{d_*}+q_j(d,\beta^d,\bm \xi^d)\right)\bm v(t)\\
&+(d_*-d){\rm diag}\left( a_j^d\beta^d\left(\eta_j+(d_*-d)\xi_j^d\right)\right)\bm v(t)\\
&+(d_*-d){\rm diag}\left( b_j^d\beta^d\left(\eta_j+(d_*-d)\xi_j^d\right)\right)\bm v(t-\tau),\\
\end{split}
\end{equation}
where $q_j(d,\beta,\bm\xi)$ and $a_j^{d}$, $b_j^{d}$ are defined in \eqref{qi} and \eqref{abd}, respectively.
Note that $a_j^d=a_j$ and $b_j^d=b_j$ for $d=d_*$, where $a_j$ and $b_j$ are defined in \eqref{r1r2}.
Then Plugging \eqref{near} into \eqref{2linearc} and removing higher order terms $\mathcal{O}(d_*-d)^2$, we see that $c(t)$ satisfies
\begin{equation}\label{cteta}
\begin{split}
c'(t)\eta_j=&\left(\ds\frac{m_j}{d_*}+q_j(d_*,\beta^{d_*},\bm \xi^{d_*})\right)\eta_jc(t)+\beta^{d_*} a_j\eta^2_jc(t)+\beta^{d_*} b_j\eta^2_jc(t-\tau),\;\; j=1,2,\cdots,
\end{split}
\end{equation}
where
$q_j(d_*,\beta^{d_*},\bm \xi^{d_*})=\beta^{d_*}(a_j+b_j)\eta_j$
by \eqref{qi}.
Multiplying \eqref{cteta} by $\va_j$ and summing these over all $j$, we see that
\begin{equation}\label{cteta2}
\begin{split}
c'(t)\sum_{j=1}^n\va_j\eta_j=&\left[\ds\frac{1}{d_*}\sum_{j=1}^nm_j\va_j\eta_j+\beta^{d_*}(\tilde a+\tilde b)\right]c(t)+\beta^{d_*} \tilde ac(t)+\beta^{d_*} \tilde b c(t-\tau)\\
=&\beta^{d_*} \tilde ac(t)+\beta^{d_*} \tilde b c(t-\tau),
\end{split}
\end{equation}
where we have used \eqref{beta*} in the first step. Therefore, removing higher order terms $\mathcal{O}(d_*-d)^2$, the linearized system \eqref{2linear} can be approximated by \eqref{cteta2}. This also explains why $\tilde a-\tilde b$ is crucial for the existence of a Hopf bifurcation.

For case (II), we also show the existence of a Hopf bifurcation, and discuss the effect of network topology on Hopf bifurcation values for a concrete model.
Our method can only apply to the case of spatial heterogeneity, since it is based on the fact that $\mathcal S_q$ is one dimensional (see Lemma \ref{l4.4}). For example, we need to impose assumption \eqref{mjmk} on model \eqref{simple} to guarantee the existence of a Hopf bifurcation. The
case of spatial homogeneity awaits further investigation.

\acknowledgements
We thank two anonymous reviewers for their insightful suggestions which
greatly improve the manuscript. We also thank Dr. Zuolin Shen for helpful suggestions on
numerical simulations.

This work is supported by the National Natural Science Foundation of China (Nos. 12171117, 11771109) and Shandong Provincial Natural Science Foundation of China (No. ZR2020YQ01).

\section*{Conflict of interest}
None.



\label{lastpage}

\end{document}